\documentclass[11pt]{amsart}
\usepackage{amssymb,latexsym,eufrak,amsmath,amscd,
graphicx}

\theoremstyle{plain}
\newtheorem{theorem}{Theorem}[section]
\newtheorem{proposition}[theorem]{Proposition}
\newtheorem{corollary}[theorem]{Corollary}
\newtheorem{lemma}[theorem]{Lemma}

\newtheorem{theoremalpha}{Theorem}
\newtheorem{corollaryalpha}[theoremalpha]{Corollary}

\theoremstyle{definition}
\newtheorem{definition}[theorem]{Definition}
\newtheorem{remark}[theorem]{Remark}

\newcommand{\lra}{\longrightarrow}
\newcommand{\noi}{\noindent}
\newcommand{\PP}{\mathbf{P}}

\newcommand{\RR}{\mathbf{R}}

\newcommand{\CC}{\mathbf{C}}
\newcommand{\QQ}{\mathbf{Q}}

\newcommand{\OO}{\mathcal  {O}}

\newcommand{\JJ}{\mathcal  {J}}

\newcommand{\fra}{\frak{a}}

\newcommand{\frb}{\frak{b}}
\newcommand{\frc}{\frak{c}}

\newcommand{\Adj}{{\rm Adj}}

\DeclareMathOperator{\vol}{vol}

\DeclareMathOperator{\ord}{ord}

\DeclareMathOperator{\Bs}{Bs}
\DeclareMathOperator{\BB}{{\bf B}}

\begin{document}

\title{Extension of sections via adjoint ideals}

\author[L. Ein]{Lawrence Ein}
\address{Department of Mathematics \\ University of
Illinois at Chicago, \hfil\break\indent  851 South
Morgan Street (M/C 249)\\ Chicago, IL 60607-7045, USA}
\email{ein@math.uic.edu}

\author[M. Popa]{Mihnea Popa}
\address{Department of Mathematics \\ University of
Illinois at Chicago, \hfil\break\indent  851 South
Morgan Street (M/C 249)\\ Chicago, IL 60607-7045, USA}
\email{mpopa@math.uic.edu}

\thanks{LE was partially supported by the NSF grant DMS-0700774. MP was partially supported by the NSF grant DMS-0601252 and by a Sloan Fellowship.}

\maketitle

\section{Introduction}

We prove some extension theorems and applications, inspired by the very interesting recent results of Hacon-M$^c$Kernan \cite{hm1}, \cite{hm2} and Takayama \cite{takayama}, used in the minimal model program and in turn inspired by fundamental results of Siu. Parts of the proofs we give follow quite closely techniques in \cite{kawamata}, \cite{hm1}, \cite{hm2}, and \cite{positivity}, which use asymptotic constructions. Related analytic statements are proved and used in Berndtsson-P\u aun \cite{bp} and 
P\u aun \cite{paun2}. Our main result is:

\begin{theoremalpha}\label{extension0}
Let $(X, \Delta)$ be a log-pair, with $X$ a normal projective variety and $\Delta$ an effective $\QQ$-divisor with $[\Delta] = 0$. Let $S \subset X$ be an irreducible normal effective Cartier divisor such that $S \not\subset {\rm Supp}(\Delta)$, and $A$ a big and nef $\QQ$-divisor on $X$ such that $S \not\subseteq \BB_{+} (A)$. Let $k$ be a positive integer such that $M = k(K_X + S + A + \Delta)$ is Cartier. Assume the following:
\begin{itemize}
\item $(X, S +  \Delta)$ is a plt pair. 
\item $M$ is pseudo-effective.
\item the restricted base locus $\BB_{-} (M)$ does not contain any irreducible closed subset $W\subset X$ 
with minimal log-discrepancy at its generic point ${\rm mld}(\mu_W ; X, \Delta +S) < 1$,
which intersects $S$ but is different from $S$ itself.
\item the restricted base locus $\BB_{-} (M_S)$ does not contain any irreducible closed subset 
$W\subset S$ with minimal log-dicrepancy at its generic point ${\rm mld}(\mu_W ; S, \Delta_S) < 1$.\footnote{By recent results on inversion of adjunction \cite{bchm}, this fourth condition is in fact superfluous, as it is implied by the third. We prefer however to work 
without using this result, especially for avoiding a circular argument in the use of the results in the last section. On the other hand, for centers on $S$ adjunction implies ${\rm mld}(\mu_W ; X, \Delta +S) \le {\rm mld}(\mu_W ; S, \Delta_S)$. Hence in the third assumption it is in fact enough to require only ${\rm mld}(\mu_W ; X, \Delta) < 1$ for centers not contained in $S$, as the rest is implied by the fourth. Same comments for all other statements.} 
\end{itemize}
Then the restriction map
$$H^0 (X, \OO_X (mM)) \longrightarrow H^0 (S, \OO_S (mM_S))$$
is surjective for all $m \ge 1$. Moreover, if $X$ is smooth, it is enough to assume that we have only $M\sim_{\QQ}  k(K_X + S + A + \Delta)$, with $M$ Cartier.
\end{theoremalpha}

To explain the notation in the statement, we start by recalling the following definitions and results from \cite{elmnp} \S1: given a $\QQ$-divisor $D$, the \emph{restricted base locus} of $D$ is defined as $\BB_{-} (D) := \bigcup \BB (D + A)$, where the union is 
taken over all ample $\QQ$-divisors $A$. We have that $\BB_{-} (D) = \emptyset$ iff $D$ is nef (thus $\BB_{-} (D)$ can be interpreted as the \emph{non-nef locus} of $M$, and it appears under this name in \cite{bdpp}), and $\BB_{-} (D) \neq X$ iff $D$ is 
pseudo-effective. The \emph{augmented base locus} of $D$ is defined as $\BB_{+} (D) := \BB (D - A)$ with $A$ any ample 
$\QQ$-divisor of sufficiently small norm. We have that $\BB_{+} (D) = \emptyset$ iff $D$ is ample, and $\BB_{+} (D) \neq X$ iff $D$ is big. There are obvious inclusions $\BB_{-} (D) \subseteq \BB (D) \subseteq \BB_{+} (D)$, where $\BB(D) : =  \bigcap_{m > 0} \Bs(mD)$ is the stable base locus of $D$. 

Recall that a plt (purely log-terminal) pair is one for which the log-discrepancy of every exceptional 
divisor is strictly greater than $0$. Therefore $(X, S+ \Delta)$ being plt in the statement above implies 
in particular that both $(X, \Delta)$ and $(S, \Delta_S)$ are klt pairs. Recall also that the \emph{minimal log-discrepancy} 
with respect to a pair $(X, \Delta)$, at the generic point of a proper irreducible closed subset $W \subset X$, is
$${\rm mld}(\mu_W ; X, \Delta): = \underset{c_X(E) = W}{\rm inf}~ {\rm ld} (E; X, \Delta),$$
where the infimum is taken over the log-discrepancies of all effective Cartier divisors $E$ on models of $X$ whose center 
$c_X(E)$ on $X$ is equal to $W$. 
Note that all $W\subset X$ such that  ${\rm mld}(\mu_W ; X, \Delta) < 1$ are among the log-canonical centers of 
$(X, \lceil \Delta\rceil)$, but the converse is often not true.

If $A$ is ample, it is $\QQ$-linearly equivalent to an effective divisor with arbitrarily low coefficient. Thus we can assume that $[A + \Delta] = 0$. Hence Theorem \ref{extension0} provides a refinement of (the singular version of) the extension result of Hacon-M$^c$Kernan, \cite{hm2} Theorem 5.4 (applied in combination with \cite{hm1} Corollary 3.17, or \cite{hm2} Theorem 6.3). The main feature is that $M$ can be assumed to be pseudo-effective, and not necessarily $\QQ$-effective, while 
the transversality condition is imposed only with respect to the restricted base locus $\BB_{-} (M)$, and only to those components of $\Delta$ which intersect $S$. Moreover, one has to consider only centers with minimal log-discrepancy 
strictly less than $1$, usually forming a proper subset among all log-canonical centers. 

Note also that if $\Delta$ has any components which intersect $S$, the third assumption in the Theorem automatically implies that $\BB_{-} (M) \neq X$, i.e. that $M$ is pseudo-effective. In the nef case we can in fact eliminate all extra 
hypotheses:

\begin{corollaryalpha}\label{nef}
In the notation of Theorem \ref{extension0}, if $K_X + S + A + \Delta$ is nef and $(X, S+ \Delta)$ is plt, then the restriction map $H^0 (X, \OO_X (mM)) \longrightarrow H^0 (S, \OO_S (mM_S))$ is surjective for all $m \ge 1$. 
\end{corollaryalpha}
\begin{proof}
If $M$ is nef, then $\BB_{-} (M) = \emptyset$, and the same holds for $M_S$, so the last three hypotheses in the statement of Theorem \ref{extension0} are automatic.
\end{proof}

This provides in particular an analogue of the Basepoint-Free Theorem for divisors of  this form, in the 
neighborhood of $S$, and globally if $S$ is ample.

\begin{corollaryalpha}
Let $(X, \Delta)$ be a log-pair, with $X$ a normal projective variety and $\Delta$ an effective $\QQ$-divisor with $[\Delta] = 0$. Let $S \subset X$ be an irreducible normal effective Cartier divisor such that $S \not\subset {\rm Supp}(\Delta)$ and 
$(X, S+ \Delta)$ is plt, and $A$ a big and nef $\QQ$-divisor on $X$ such that $S \not\subseteq \BB_{+} (A)$. If $K_X + S + A + \Delta$ is nef, then it is semiample in a neighborhood of $S$. If in addition
$S$ is ample, then $K_X + S + A + \Delta$ is semiample.
\end{corollaryalpha}
\begin{proof}
If $k$ is a positive integer such that $M = k(K_X + S + A + \Delta)$ is Cartier, we can write 
$M_S = k(K_S + A_S + \Delta_S) = K_S + (k-1)(K_S + A_S + \Delta_S) + A_S + \Delta_S$.
Note that the pair $(S, \Delta_S)$ is klt, and the assumptions imply that $M_S - K_S - \Delta_S$ is nef and big. The Basepoint-Free Theorem (cf. e.g. \cite{km} Theorem 3.3) implies then that $M_S$ is semiample. Since by Corollary \ref{nef} all sections of multiples of $M_S$ lift, this gives the semiampleness of $M$ in a neighborhood of $S$. Now if $\BB( M) \neq \emptyset$, it must be 
positive dimensional by a well-known result of Zariski (see \cite{positivity} Remark 2.1.32).  
In case $S$ is ample, this implies that 
$S$ intersects $\BB(M)$ nontrivially,  a contradiction with the above.
\end{proof}

The situation turns out to be very special when $k=1$. In this case many of the assumptions in Theorem \ref{extension0}, especially that of pseudo-effectivity, can be dropped, and 
we obtain with the same (simplified) method a proof of Takayama's extension theorem \cite{takayama}, Theorem 4.1. We present this approach to Takayama's theorem as a toy model for the general proof at the beginning of \S5. 

The case $\Delta = 0$  and $X$ smooth in Theorem \ref{extension0} is particularly instructive. It clarifies how various results mentioned above can be strengthened (there is no transversality hypothesis for $S$), 
but at the same time it shows that one cannot hope for removing the a priori pseudo-effectivity hypothesis as in Takayama's theorem, as soon as one passes from Cartier divisors to arbitrary $\QQ$-divisors (cf. the Example in \S5).

\begin{corollaryalpha}
Let $X$ be a smooth projective variety,  $S\subset X$ a smooth divisor, and $A$ a big and nef $\QQ$-divisor on $X$. Assume that $K_X + S+ A$ is pseudo-effective, and let $k$ be an integer and $M$ a Cartier divisor such that $M \sim_{\QQ} k(K_X + S + A)$. Then, for all $m\ge1$, the restriction map
$$H^0 (X, \OO_X (mM)) \longrightarrow H^0 (S, \OO_S (mM_S))$$
is surjective.
\end{corollaryalpha}

Theorem \ref{extension0} has as an immediate corollary another strengthening of an extension result of Hacon-M$^c$Kernan, namely \cite{hm1} Corollary 3.17. We state it separately, especially for the particularly clean statement at the end. In addition to the improvements mentioned above, note that it is enough to add any ample divisor $H$, and \emph{all} sections of $mM + H$ can be lifted.

\begin{corollaryalpha}\label{extension2}
Let $(X, \Delta)$ be a log-pair, with $X$ a normal projective variety and $\Delta$ an effective $\QQ$-divisor with $[\Delta] = 0$. Let $S \subset X$ be an irreducible normal effective Cartier divisor such that $S \not\subset {\rm Supp}(\Delta)$. 
Let $B$ be a nef $\QQ$-divisor. Define $M := K_X + S + B + \Delta$. Let $H$ be any ample $\QQ$-divisor on $X$, and let $m\ge 1$ be any integer such that $mM + H$ is Cartier. Assume the following:
\begin{itemize}
\item $(X, S +  \Delta)$ is a plt pair. 
\item $M$ is pseudo-effective.
\item the restricted base locus $\BB_{-} (M)$ does not contain any irreducible closed subset $W\subset X$ 
with ${\rm mld}(\mu_W ; X, \Delta+S) < 1$ which intersects $S$. 
\item the restricted base locus $\BB_{-} (M_S)$ does not contain any irreducible closed subset 
$W\subset S$ 
with ${\rm mld}(\mu_W ; S, \Delta_S) < 1$.
\end{itemize}
Then the restriction map
$$H^0 (X, \OO_X (mM+H)) \longrightarrow H^0 (S, \OO_S (mM_S + H_S))$$
is surjective.
\end{corollaryalpha}
\begin{proof}
Fix an $m$ as in the statement. Write 
$$mM + H = m(K_X + S + A_m + \Delta),$$
where $A_m : = B + \frac{1}{m} H$ is an ample $\QQ$-divisor. 
Since $H$ is ample, $mM + H$ is $\QQ$-effective.  On the other hand, by definition we have 
$$\BB (mM_S + H_S) \subseteq \BB_{-} (M_S).$$ 
Thus all the hypotheses in Theorem \ref{extension0} are satisfied (we are replacing $k$ by $m$, and $m$ by $1$).
\end{proof}

\begin{remark}\label{simplification}
Again, if $M$ is already Cartier, the assumption that $M$ be pseudo-effective can be dropped. This 
follows from the proof of Takayama's result, Theorem \ref{takayama} below. 
\end{remark}

Some cases of the results above have been proved under weaker positivity hypotheses
by analytic methods. For instance, Takayama's result under weaker positivity was proved by Varolin \cite{varolin} (using ideas of Siu \cite{siu2} and P\u aun \cite{paun1}, where it is of course also 
shown that standard invariance of plurigenera holds under very general assumptions). We would also 
like to mention the nice recent preprint of Berndtsson-P\u aun \cite{bp}, where the authors use a special case of Theorem \ref{extension0} for a different view on subadjunction, and provide an analytic approach to it.

Finally, a relative version of Theorem \ref{extension0}, stated in \S6,  can be used to prove a general deformation-invariance-type statement. It is important to note part (ii) below: under resonable assumptions one does not require the a priori pseudo-effectivity 
of $M$ any more. In the case when $\Delta = 0$ and $A$ is integral, this is then a weak version of a well-known theorem due to Siu \cite{siu2} (cf. also \cite{paun1}). Besides Theorem \ref{extension0}, a main ingredient is the continuity of the volume function associated to pseudo-effective divisors.

\begin{theoremalpha}\label{fibration}
Let $\pi : X \rightarrow T$ be a projective morphism with normal fibers from a normal variety $X$ to a smooth curve $T$. Let $\Delta$ be a horizontal effective $\QQ$-divisor on $X$ such that $[\Delta] = 0$. Let $A$ be a $\pi$-big and nef $\QQ$-divisor and $k$ a positive integer such that $M=  k(K_X + A + \Delta)$ is Cartier. Denote by $X_t$ the fiber of $\pi$ over $t\in T$, and assume for all $t$ the following:
\begin{itemize}
\item $(X, X_t + \Delta) $ is a plt pair. 
\item $M$ is pseudo-effective.
\item the restricted base locus $\BB_{-} (M)$ does not contain any irreducible closed subset $W\subset X$ 
with ${\rm mld}(\mu_W ; X, \Delta) < 1$ which intersects $X_t$. 
\item the restricted base locus $\BB_{-} (M_{X_t})$ does not contain any irreducible closed subset 
$W\subset X_t$ 
with ${\rm mld}(\mu_W ; X_t, \Delta_{X_t}) < 1$.
\end{itemize}
Then: (i) the dimension of the space of global sections $H^0 (X_t, \OO_X (mM_{X_t}))$ is constant for $t \in T$, 
for all $m\ge 1$. Moreover, if $X$ is smooth, it is enough to assume that we have only $M\sim_{\QQ}  k(K_X + A + \Delta)$, with $M$ Cartier.

\noindent
(ii) If in addition the pair $(X_t, k\Delta_{X_t})$ is klt for all $t \in T$, then the hypothesis that $M$ be pseudo-effective is not necessary.
\end{theoremalpha}

This has consequences (cf. Corollary \ref{pseudoeffective} and Corollary \ref{boundary}) to the invariance of pseudo-effectivity in families, hence of the Kodaira energy and of part of the boundary of the pseudo-effective cone,  for $\QQ$-divisors of the form treated here. This is analogous to the fact that Siu's original extension results give the invariance of the Kodaira energy for semipositive divisors.
Theorem \ref{fibration} was recently used by Totaro \cite{totaro} in proving that the Cox ring deforms in a flat family under a deformation of a terminal Fano which is $\QQ$-factorial in codimension 3.

The most beautiful recent application of extension theorems has been to the problem of finite generation and of existence of flips 
(cf. Hacon-M$^c$Kernan \cite{hm2}, \cite{hm3}). In \S7 we show how a quick argument based on Theorem \ref{extension0}, on Takayama's technique, and still of course on ideas of Hacon-M$^c$Kernan, gives part of the finite generation statements found in the recent \cite{hm3}, needed for the existence of flips. Namely, in order to have finite generation one only needs to know that certain asymptotic vanishing orders are rational, but not necessarily achieved.\footnote{To further deal with this rationality issue, Hacon and M$^c$Kernan use diophantine approximation and an inductive MMP argument based on \cite{bchm}.} Suffice it to say here that these methods imply that without
any transversality assumptions as in Theorem \ref{extension0}, one can still lift all sections vanishing to sufficiently high order along the components of $\Delta$. The key statement is, loosely speaking (essentially 
after reductions to the smooth simple normal crossings case): \emph{Say $M$ is as in Theorem \ref{extension0}, with $A$ ample, and denote
$\delta_i : =  {\rm ord}_{F_i} (\Delta_S)$ and $c_i :=  \ord_{F_i} \parallel \frac{M}{k} \parallel_S$ (the asymptotic order of vanishing of the ``restricted" linear series of sections coming from $X$), where 
$F_i$ are all the intersections of components of the support of $\Delta$ with $S$.
Then, assuming that $S \not\subseteq \BB(M)$, but without any assumptions on $\BB_{-} (M)$ or $\BB_{-} (M_S)$, all sections in $H^0 (S, \OO_S(mM_S))$ vanishing to order at least ${\rm min}\{\delta_i, (1-\epsilon)c_i\}$ along $F_i$ can be lifted to $X$, where $\epsilon$ is any real number such that
$A  + mk\epsilon (K_X + S + \Delta)$ is ample.} 
For precise statements cf. \S7, especially Theorem \ref{rational}.

The proofs are based on a systematic study of certain \emph{adjoint ideals}, and of their asymptotic counterparts, defined in \S2 and \S3 respectively. They are a mix between the multiplier ideals associated to effective $\QQ$-divisors (cf. \cite{positivity}, Part III) and the adjoint ideals associated to reduced irreducible (and in fact, more generally, simple normal crossings) effective divisors (cf. \cite{positivity} 9.3.E). These ideals have certainly appeared in literature in various forms,
and especially in \cite{hm1} in the form used here. Without any claim to originality regarding the definition, we develop a general inductive method for understanding their vanishing and extension properties. This should hopefully be useful in many other situations. 

The main strategy in this picture originates of course in Siu's fundamental work. The use of adjoint ideals for extension problems in this paper is inspired by, and at times very similar to, techniques in \cite{kawamata} and \cite{hm1} using asymptotic multiplier ideals, but differs somewhat from the approach in \cite{takayama} (following Siu), where sections are lifted individually.
We prove extension after passing to a suitable birational model, following ideas of 
Hacon-M$^c$Kernan in their proof of the existence of flips \cite{hm2}. In the recent \cite{hm3} the authors provide a new proof, based on improved extension techniques which certainly have some overlap with our results. Finally, similar extension techniques are used in the recent papers
of Berndtsson-P\u aun \cite{bp} on subadjunction and of P\u aun \cite{paun2} on non-vanishing.
For the benefit of the reader, an outline of the method is given before the proof of Theorem 
\ref{extension0} in \S5.

\noindent
{\bf Acknowledgements.} 
We would like to thank Ch. Hacon, R. Lazarsfeld, V. Lazic, J. M$^c$Kernan, M. Musta\c t\u a and  M. P\u aun for interesting discussions. We are especially grateful to R. Lazarsfeld and M. Musta\c t\u a for suggestions related to Theorem \ref{fibration}.

\section{Adjoint ideals}

Let $X$ be a smooth projective complex variety. We study formally a notion of \emph{adjoint ideal} on $X$, which is roughly speaking a combination of a multiplier ideal attached to an ideal sheaf and an adjoint ideal attached to a simple normal crossings divisor. It  appears to some extent in various places in the literature, but the form we adopt here is that of \cite{hm1}. For the general theory of multiplier ideal sheaves, including asymptotic constructions, we refer the reader to \cite{positivity}, Chapters $9$
and $11$.

\subsection*{Restriction to an irreducible divisor.}
To explain how this works, we start with the simplest (and quite well-known) case, which is particularly transparent: let $\fra \subseteq \OO_X$ be an ideal sheaf, and $D \subset X$ a smooth integral divisor such that $D \not \subset Z(\fra)$. Fix also $\lambda \in \QQ$. 

Let $f: Y \rightarrow X$ be a common log-resolution for the pair $(X, D)$ and the ideal $\fra$, and write $\fra \cdot \OO_{Y} = \OO_{Y} (-E)$. By construction the union of $E$ and the proper transform $\tilde D$ is in simple normal crossings. 

\begin{definition}
The \emph{adjoint ideal} associated to the data $(X, D, \fra, \lambda)$ is 
$$\Adj_D ( X, \fra^{\lambda}) : = f_* \OO_{Y} (K_{Y / X} - f^*D + \tilde D - [\lambda E]).$$
One can check as for multiplier ideals (cf. \cite{positivity}, Theorem 9.2.18) that this definition is independent on the choice of log-resolution. It is clear that 
$$\Adj_D ( X, \fra^{\lambda})\subset \JJ (X, \fra ^{\lambda}),$$ 
the multiplier ideal associated to $\fra$ and $\lambda$, so in particular it is an ideal sheaf. 
\end{definition}

\noindent
On the log-resolution $Y$ we have the standard exact sequence
$$0\longrightarrow \OO_Y(K_{Y /X} - f^*D - [\lambda E]) \longrightarrow \OO_Y(K_{Y /X} - f^*D - [\lambda E] + \tilde D) \longrightarrow $$
$$ \longrightarrow \OO_Y(K_{Y /X} - f^*D - [\lambda E] + \tilde D)_{|\tilde D} \longrightarrow 0.$$
We have the following: 
\begin{itemize}
\item  $f_* \OO_Y(K_{Y /X} - f^*D - [\lambda E]) \cong \JJ(X, \fra^{\lambda}) (-D)$ by the projection formula and the definition of multiplier ideals. 
\item $R^i f_* \OO_Y(K_{Y /X} - f^*D - [\lambda E]) = 0, ~\forall i>0$, by Local Vanishing (cf. \cite{positivity} 9.4.4).
\item $\OO_Y( -[\lambda E]) _{|\tilde D} \cong \OO_{\tilde D} (- [\lambda(E\cap \tilde D)])$, since 
$E \cup \tilde D$ is in simple normal crossings. 
\end{itemize}
Using these facts and adjunction, by push-forward we obtain a basic exact sequence
\begin{equation}\label{sequence_1}
0 \longrightarrow \JJ(X, \fra^{\lambda}) (-D) \longrightarrow \Adj_D ( X, \fra^{\lambda}) \longrightarrow \JJ (D, (\fra \cdot \OO_D)^{\lambda})\longrightarrow 0.
\end{equation}
Note that Local Vanishing also applies to the right hand side by base-change, so we have its version for the adjoint ideal as well: 
$$R^i f_* \OO_Y(K_{Y /X} - f^*D - [\lambda E] + \tilde D) = 0, ~\forall i>0.$$
Finally the sequence (\ref{sequence_1}), together with Nadel Vanishing, provides 
a vanishing theorem for the adjoint ideal, namely
$$H^i (X, \OO_X(K_X + L + D) \otimes  \Adj_D ( X, \fra^{\lambda})) = 0, ~\forall i>0,$$
where $L$ is any line bundle on $X$ such that $L- \fra^{\lambda}$ big and nef, and $D\not\subset \BB_{+} (L - \fra^{\lambda})$ (cf. 
Definition \ref{variation}).

\begin{remark}\label{divisors}
These constructions and results have, as usual, variants involving $\QQ$-divisors as opposed to 
ideals. Let $D$ be a smooth integral divisor on $X$, and let $B$ be an effective $\QQ$-divisor such that $D \not \subset {\rm Supp} (B)$. Then we can define
$$\Adj_D ( X,  B) := f_* \OO_{Y} (K_{Y / X} - f^*D + \tilde D - [f^* B]),$$
where $f:Y \rightarrow X$ is a log-resolution for $(X, D+ B)$.
This sits in an exact sequence\footnote{This sequence is a small variation of the exact sequence describing the adjoint 
ideal of a divisor in \cite{el2} (cf. also \cite{positivity} 9.3.E). It was considered in \cite{takayama} as well.} 
\begin{equation}\label{sequence}
0 \longrightarrow \JJ(X, B) (-D) \longrightarrow \Adj_D ( X,  B) \longrightarrow \JJ (D, B_{|D})\longrightarrow 0
\end{equation}
and again satisfies Local Vanishing. The vanishing theorem it satisfies is:
$$H^i (X, \OO_X(K_X + L + D) \otimes  \Adj_D ( X,  B)) = 0, ~\forall i>0,$$
where $L$ is any line bundle on $X$ such that $L - B$ is big and nef and $D \not\subset \BB_{+} (L-B)$. (This last thing
means that $(L-B)_{|D}$ is still big, which is how we always make use of $\BB_{+}$.)
\end{remark}

\begin{remark}[Adjoint ideals on pairs]\label{pair_1}
As in \cite{positivity} \S9.3.G in the case of multiplier ideals, we can also make these constructions live on a pair $(X, \Lambda)$, with $\Lambda$ an effective $\QQ$-divisor. If we impose the condition that $D\not\subset B \cup \Lambda$, we can define as above 
$$\Adj_D  ((X, \Lambda); B ) := f_* \OO_{Y} (K_{Y / X} - f^*D + \tilde D - [f^*( B+ \Lambda)]).$$
Clearly since $X$ is smooth this is not something new; in fact:
$$\Adj_D ((X, \Lambda); B\big) = \Adj_D  (X, B + \Lambda).$$ 
However this helps with the notation: if $\fra$ is an ideal sheaf satisfying the conditions above, we can define similarly
$\Adj_D  ((X, \Lambda); \fra^{\lambda})$. All the results above have obvious analogues in this context.
\end{remark}

\subsection*{Restriction to a reduced simple normal crossings (SNC) divisor.}
Let $\Gamma$ be a reduced SNC divisor on $X$, and $\fra\subseteq \OO_X$ an ideal sheaf such that 
no log-canonical center of $\Gamma$ is contained in $Z(\fra)$.\footnote{Given the SNC condition, this simply means the intersections of the various components of $\Gamma$, including the components themselves.}  Let $f: Y \rightarrow X$ be a common log-resolution for the pair $(X, \Gamma)$ and the ideal $\fra$, and write $\fra \cdot \OO_{Y} = \OO_{Y} (-E)$. By construction the union of $E$ and the proper transform $\tilde \Gamma$ is SNC. 

\begin{definition}
The \emph{adjoint ideal} associated to the data $(X, \Gamma, \fra, \lambda)$ is 
$$\Adj_{\Gamma} ( X, \fra^{\lambda}) : = f_* \OO_{Y} (K_{Y / X} - f^*\Gamma 
+ \underset{{\rm ld}(\Gamma; D_i)=0}{\sum} D_i - [\lambda E]),$$
where the sum appearing in the expression is taken over all divisors on $Y$ having log-discrepancy 
$0$ with respect to $\Gamma$, i.e. among those appearing in $\Gamma^{\prime}$ in the expression $K_Y + \Gamma ^{\prime}  = f^* (K_X + \Gamma)$. Note again that
$$\Adj_{\Gamma} ( X, \fra^{\lambda})\subset \JJ(X, \fra^{\lambda}).$$
\end{definition}
The next Lemma implies that adjoint ideals are independent of the choice of log-resolution.

\begin{lemma}\label{independence}
Let $f:Y \rightarrow X$ be a log-resolution for $(X, \Gamma, \fra)$, with $\fra \cdot \OO_{Y} = \OO_{Y} (-E)$, and let $g: Z\rightarrow Y$ be another birational morphism, with $Z$ smooth, $\fra \cdot \OO_{Z} = \OO_{Z} (-E^{\prime})$, such that the proper transform of $\Gamma$, $E^{\prime}$, and the exceptional divisor of $g$ form an SNC divisor. Then 
$$f_* \OO_{Y} (K_{Y / X} - f^*\Gamma 
+ \underset{{\rm ld}(\Gamma; D_i)=0}{\sum} D_i - [\lambda E]) \cong (f\circ g)_* \OO_Z(K_{Z / X} - (f\circ g)^*\Gamma 
+ \underset{{\rm ld}(\Gamma; D_i^{\prime})=0}{\sum} D_i - [\lambda E^{\prime}]),$$
where the $D_i^{\prime}$ are the divisors on $Z$ with log-discrepancy $0$ with respect to $\Gamma$.
\end{lemma}
\begin{proof}
The Lemma follows if we prove that the difference between $K_{Z / X} - (f\circ g)^*\Gamma 
+ \underset{{\rm ld}(\Gamma; D_i^{\prime})=0}{\sum} D_i - [\lambda E^{\prime}]$ and 
$g^*(K_{Y / X} - f^*\Gamma 
+ \underset{{\rm ld}(\Gamma; D_i)=0}{\sum} D_i - [\lambda E])$ is effective and exceptional. Note that $g^* E = E^{\prime}$, and also that in general $g^* \underset{{\rm ld}(\Gamma; D_i)=0}{\sum} D_i \ge \underset{{\rm ld}(\Gamma; D_i^{\prime})=0}{\sum} D_i^{\prime}$. If we decompose $[\lambda E] = \lambda E - \{\lambda E\}$ and 
$[\lambda E^{\prime}] = \lambda E^{\prime} - \{\lambda E^{\prime}\}$, it follows that what we need to prove is that 
$$K_{Z/X} + \{\lambda E^{\prime}\}- g^*\{\lambda E\}  - g^* \underset{{\rm ld}(\Gamma; D_i)=0}{\sum} D_i + \underset{{\rm ld}(\Gamma; D_i^{\prime})=0}{\sum} D_i^{\prime}$$
is effective and exceptional. But since we know that the total expression is an integral 
divisor, it is enough to show that 
\begin{equation}\label{sum}
K_{Z/X} - [g^*\{\lambda E\}]  - g^* \underset{{\rm ld}(\Gamma; D_i)=0}{\sum} D_i + \underset{{\rm ld}(\Gamma; D_i^{\prime})=0}{\sum} D_i^{\prime}
\end{equation}
is effective and exceptional. Note that we know already that $K_{Z/X} - [g^*\{\lambda E\}]$ is so. (This is 
essentially  the independence of the log-resolution for the usual multiplier ideals, cf. \cite{positivity} Lemma 9.2.19.) It is clear then that the only thing we need to check 
is that the $D_i^{\prime}$ which are not pull-backs of some $D_i$ appear in this sum with non-negative coefficient. 

To this end, fix one such $D_i^{\prime}$. It is $g$-exceptional, and denote by $Z_i^{\prime}$ its center on $Y$, i.e. $Z_i^{\prime} := g(D_i^{\prime}) \subset Y$, say of 
codimension $d$. Denote by $E_1, \ldots, E_a$ the components of $E$, and by $D_1, \ldots, D_b$ the $D_i$'s containing $Z_i^{\prime}$. Since everything is SNC, we can also 
choose $F_1, \ldots, F_c$ prime divisors, with $a + b + c = d$, such that  $E_1, \ldots, 
E_a, D_1, \ldots, D_b, F_1, \ldots, F_c$ form a system of parameters at the generic point of $Z_i^{\prime}$. Choose local equations $x_k$ for $E_k$, $y_l$ for $D_l$, and 
$z_m$ for $F_m$, at this generic point. Denote also by $(t=0)$ the local equation of 
$F_i^{\prime}$. 

By our choice, there exist positive integers $s_k$, $r_l$ and $p_m$, and invertible elements $u_k$, $v_l$ and $w_m$, such that at the generic point we have 
$$x_k = t^{s_k}\cdot u_k, ~~y_l = t^{r_l}\cdot v_l, {\rm~and~} z_m = t^{p_m}\cdot w_m.$$
Denoting by $k_i$ the coefficient of $F_i^{\prime}$ in $K_{Z/X}$, a local calculation precisely as in the proof of \cite{positivity} Lemma 9.2.19, shows immediately that
$$k_i \ge \sum s_k + \sum r_l + \sum p_m - 1.$$ 
Going back to $(\ref{sum})$, we see that the coefficient of $D_i^{\prime}$ there is then 
at least 
$$\sum s_k - [\sum \lambda_ks_k]  -1,$$
where $0 \le \lambda_k <1$ are the coefficients corresponding to the $E_k$ that appear in the sum. Since we are assuming that $D_i^{\prime}$ is not the pull-back of some $D_i$, there is at least one $E_k$ in the sum, so the expression is non-negative.
\end{proof}

Adjoint ideals can be studied inductively, based on the number of components of $\Gamma$. To this end, let $S\subset X$ be a smooth divisor such that $\Gamma + S$ is also SNC, and assume that no log-canonical center of $\Gamma + S$ is contained in $Z(\fra)$. 

\begin{proposition}\label{basic_sequence_1}
With the notation above, there is a short exact sequence of ideal sheaves
\begin{equation}\label{induction}
0 \longrightarrow \Adj_{\Gamma} ( X, \fra^{\lambda}) \otimes \OO_X(-S) \longrightarrow 
\Adj_{\Gamma + S}  (X, \fra^{\lambda}) \longrightarrow 
\Adj_{\Gamma_S}  (S, \fra^{\lambda}\cdot \OO_S)\longrightarrow 0.
\end{equation}
\end{proposition}
\begin{proof}
By Lemma \ref{independence}, the definition of adjoint ideals is independent of the choice of log-resolution. 
Given the assumption in the definition, we can form a 
log-resolution by blowing up smooth centers not containing any of the log-canonical 
centers of $(X, \Gamma + S)$, so that upstairs the only divisors with log-discrepancy $0$ 
with respect to $\Gamma$ or $\Gamma + S$ are the components of their proper transforms.
We first fix such a resolution $f:Y \rightarrow X$ for $(X, \Gamma+ S, \fra)$ in the argument.\footnote{Such a log-resolution is called a \emph{canonical} resolution in \cite{hm1}.} We make the following \emph{claim} (Local Vanishing for 
adjoint ideals):
\begin{equation}\label{loc_van}
R^i f_* \OO_{Y} (K_{Y / X} - f^*\Gamma 
+ \underset{{\rm ld}(\Gamma; D_i)=0}{\sum} D_i - [\lambda E])= 0, {\rm ~for~all~} i>0.
\end{equation}

We prove both the claim and the statement of the Proposition at the same time, by induction on the number of components of $\Gamma$. If $\Gamma$ is irreducible, then ($\ref{loc_van}$) follows from Local Vanishing for multiplier ideals, as explained in the previous subsection. This means that by induction we can always assume that we have Local Vanishing for the two extremes in ($\ref{induction}$). 

Denote $\fra\cdot \OO_Y = \OO_Y(-E)$. On $Y$, the sum of divisors which have log-discrepancy $0$ with respect to $\Gamma + S$ is 
$\tilde S + \underset{{\rm ld}(\Gamma; D_i)=0}{\sum} D_i$. We can then form an exact sequence as follows: 
$$0 \rightarrow \OO_Y( K_{Y/X} - [\lambda E] - f^*\Gamma - f^*S ~ + \underset{{\rm ld}(\Gamma; D_i)=0}{\sum} D_i) \rightarrow 
\OO_Y( K_{Y/X} - [\lambda E] - f^*(\Gamma + S) + \tilde S + \underset{{\rm ld}(\Gamma; D_i)=0}{\sum} D_i)$$
$$ \longrightarrow 
\OO_{\tilde S}(K_{Y/X} - [\lambda E] - f^*\Gamma - f^*S + \tilde S + \underset{{\rm ld}(\Gamma; D_i)=0}{\sum} D_i) \longrightarrow 0.$$
Since $E$ and $\tilde S$ are in simple normal crossings, we
have that $[\lambda E]_{\tilde S} = [\lambda E_{\tilde S}]$. Furthermore $\tilde S$ is a common log-resolution for $\Gamma_S$ and $\fra\cdot \OO_S$, and by adjunction 
$$(K_{Y/X} - f^*S + \tilde S)_{|\tilde S} \cong K_{\tilde S /S}.$$ 
Therefore the terms in the exact sequence above push forward to the terms of the sequence appearing in ($\ref{induction}$), by definition. Now the statement in ($\ref{induction}$) follows by pushing forward via $f$ and applying ($\ref{loc_van}$), which we assumed to know inductively for $\Gamma$ and $\Gamma_S$. This also implies local vanishing for the middle term, so it proves ($\ref{loc_van}$) too for 
this special resolution.

On the other hand, knowing ($\ref{induction}$) we can deduce ($\ref{loc_van}$) by the same inductive argument for any resolution, since the base case is Local Vanishing for multiplier ideals, known to hold independently of resolution. 
\end{proof}

\begin{remark}\label{pair_2}
As in Remark \ref{pair_1}, the entire discussion above goes through for adjoint ideals defined on a pair. Indeed, if $(X, \Lambda)$ is a pair with $\Lambda$ an  effective $\QQ$-divisor such that no log-canonical center of $\Gamma$ is contained in $Z(\fra) \cup {\rm Supp}(\Lambda)$, then we can define 
$\Adj_{\Gamma} ((X,\Lambda); \fra^{\lambda})$
and an obvious analogue of Proposition \ref{basic_sequence_1} holds.
\end{remark}

\begin{definition}\label{variation}
Let $D$ be a $\QQ$-divisor and $\fra$ an ideal sheaf on $X$, and let $\lambda\in \QQ$. We say that $D - \fra^{\lambda}$ is nef and/or big 
if the following holds: consider the blow-up $f: Y \rightarrow X$ along $\fra$, and denote $\fra\cdot \OO_Y = \OO_Y(-E)$. Then 
$f^* L - \lambda E$ is nef and/or big. The is easily seen to be equivalent to the same condition on any log-resolution of $(X, \fra)$.
We define in a similar way $\BB_{+} (D - \fra^{\lambda})$.
\end{definition}

\begin{theorem}[Vanishing for adjoint ideals]\label{adjoint_nadel}
With the same notation, let $L$ be a line bundle on $X$ such that $L- \Lambda -\fra^{\lambda}$ is big and nef, and no log-canonical centers of 
$(X, \Gamma)$ are contained in $\BB_{+} (L - \Lambda - \fra^{\lambda})$. Then 
$$H^i (X, \OO_X(K_X + L + \Gamma)\otimes \Adj_{\Gamma}  ((X, \Lambda); \fra^{\lambda})) = 0, 
{\rm~for~all~} i>0.$$
\end{theorem}
\begin{proof}
This follows by induction on the number of components of $\Gamma$, using Proposition \ref{basic_sequence_1} and Remark \ref{pair_2}. If $\Gamma$ is irreducible, it is a consequence of the Nadel vanishing theorem for multiplier ideals, as explained in the 
previous subsection. Assume then that we have vanishing for $\Gamma$, and we want to pass to $\Gamma + S$, with $S$ smooth such that $\Gamma + S$ is SNC. We twist the exact sequence in Proposition \ref{basic_sequence_1} by $\OO_X(K_X + L + \Gamma + S)$ and pass to cohomology. If $i > 0$, the $H^i$ on the left vanish by induction. On the 
other hand, by adjunction, on the right hand side we are looking at the cohomology groups
$$H^i (S, \OO_S (K_S + L_S + \Gamma_S) \otimes \Adj_{\Gamma_S}  ((S, \Lambda_S); (\fra\cdot \OO_S)^{\lambda}) ).$$ 
But by assumption $L_S - \Lambda_S - \fra^{\lambda}\cdot \OO_S$ is big and nef, and stays big and nef on all the log-canonical centers of $(S, \Gamma_S)$. So these groups also vanish by induction.  
\end{proof}

\begin{remark}
As usual, for everything in this subsection there is a corresponding statement involving 
divisors instead of ideal sheaves. 
\end{remark}

\section{Asymptotic adjoint ideals}

In this section we define an asymptotic version of the adjoint ideals studied in the previous section. Let $X$ be a smooth projective variety, and $\Gamma\subset X$ a reduced SNC divisor. 

\noindent
{\bf Graded systems of ideals.} Consider first a graded system $\fra_{\bullet} = \{\fra_m\}$ of ideal sheaves on $X$ (cf. \cite{positivity} 11.1.B). Assume that  no log-canonical center of $\Gamma$ is contained in $Z(\fra_m)$ for $m$ sufficiently divisible. Then for any $\lambda \in \QQ$, one can define the \emph{asymptotic adjoint ideal} as
$$\Adj_{\Gamma} (X, \fra_{\bullet}^{\lambda}) := \Adj_{\Gamma} (X, \fra_{m}^{\lambda/m}), 
~{\rm for~} m {\rm~sufficiently~divisible}.$$
The correctness of the definition can be checked exactly as in the case of asymptotic multiplier ideals, cf. \cite{positivity} \S11.1.

\noindent
{\bf Complete linear series.}
Let $L$ be a line bundle on $X$ with $\kappa(L) \ge 0$, and assume that no log-canonical center of $\Gamma$ is contained in the stable base locus $\BB(L)$. For any $m \ge 1$, denote  by $\frb_m$ the base-ideal of $|mL|$, so that $\BB(L)$ is set-theoretically $Z(\frb_m)$ for $m$ sufficiently divisible. In this case $\frb_{\bullet} := \{\frb_m\}$ is a graded system, and the asymptotic adjoint ideal of $L$ is defined as 
$$\Adj_{\Gamma} (X, \parallel L\parallel) := \Adj_{\Gamma} (X, \frb_{\bullet}).$$

\noindent
{\bf Restricted linear series.}
Let now $L$ be a line bundle on $X$, with $\kappa(L) \ge 0$, and let $Y$ be a smooth 
subvariety of $X$. Define on $Y$ the ideal sheaves $\frc_m : = \frb_m\cdot \OO_Y$, where as above $\frb_m$ is the base-ideal of the linear series $|mL|$ on $X$. The ideal $\frc_m$ is the base-ideal of the linear subseries of $|mL_Y|$ consisting of  the divisors which are restrictions of divisors on $X$, and 
$\frc_{\bullet} : = \{\frc_m\}$ is a graded system. Consider also a reduced SNC divisor $\Gamma_Y$ on $Y$ such that no log-canonical center of $\Gamma_Y$ is contained in $Z(\frc_m)$ for $m$ sufficiently divisible.
We define the \emph{restricted} asymptotic adjoint ideal as 
$$\Adj_{\Gamma_Y} (Y,  \parallel L\parallel_{|Y}) := \Adj_{\Gamma_Y} (Y, \frc_{\bullet}).$$

\noindent
{\bf Pairs.} 
It is convenient to use also the language of adjoint ideals defined on pairs $(X, \Lambda)$,
rather that just on $X$. With the notation above, one can define analogously
$$\Adj_{\Gamma} ((X, \Lambda); \fra_{\bullet}^{\lambda}), ~~
\Adj_{\Gamma} ((X,\Lambda); \parallel L\parallel) 
 {\rm~~and~~} \Adj_{\Gamma_Y} ((Y,\Lambda_Y); \parallel L\parallel_{|Y}) .$$
For this definition we have to impose the condition that 
no log-canonical center of $\Gamma$ is contained in $\BB(L) \cup {\rm Supp}(\Lambda)$ (and the corresponding conditions with respect to $Y$).

\medskip 

The analogue, in the asymptotic case, of the exact sequence in Proposition \ref{basic_sequence_1} is given below. We state it for convenience in the case of linear series. The proof is an immediate modification of that argument, and we do not repeat it here.

\begin{proposition}\label{basic_sequence_2}
Let $X$ be a smooth projective variety, $L$ a line bundle on $X$ with $\kappa(L) \ge 0$, and $S$ a smooth divisor on $X$. Consider also pairs $(X, \Lambda)$ and $(X, \Gamma)$, with $\Lambda$ an effective $\QQ$-divisor and $S + \Gamma$ a reduced SNC divisor. Assume that no log-canonical center of $S + \Gamma$ is contained in $\BB(L) \cup 
{\rm Supp}(\Lambda)$. Then there is a short exact sequence of ideal sheaves
$$0 \longrightarrow \Adj_{\Gamma} ((X, \Lambda); \parallel L \parallel) \otimes \OO_X(-S) \longrightarrow 
\Adj_{\Gamma+ S} ((X, \Lambda); \parallel L \parallel) \longrightarrow $$
$$\longrightarrow 
\Adj_{\Gamma_S} ((S, \Lambda_S); \parallel L \parallel_{|S})\longrightarrow 0.$$
\end{proposition}

\begin{theorem}[Vanishing for asymptotic adjoint ideals]\label{asymptotic_nadel}
With the same notation as above, assume that $L$ is big and that $A$ is a Cartier divisor on $X$ such that $A - \Lambda$ is nef.  If no log-canonical center of $(X, \Gamma)$ is contained in $\BB_{+} (L)$, then 
$$H^i (X, \OO_X(K_X + L  + A + \Gamma)\otimes \Adj_{\Gamma} ((X, \Lambda); \parallel L \parallel)) = 0,  {\rm~for~all~} i>0.$$
\end{theorem}
\begin{proof}
The proof is similar to that of Theorem \ref{adjoint_nadel}. The desired vanishing is reduced inductively 
(over the number of components of $\Gamma$) to Nadel vanishing for asymptotic multiplier ideals, via the 
exact sequence in Proposition \ref{basic_sequence_2}. The only significant thing to note is that in the 
asymptotic case $A - \Lambda$ can be assumed to be only nef (as opposed to nef and big). This follows from the similar fact for asymptotic multiplier ideals, \cite{positivity} Theorem 11.2.12(ii).
\end{proof}

\begin{remark}\label{no_restriction} 
In case $\Gamma = 0$, the adjoint ideal $\Adj_{\Gamma} ((X, \Lambda); \parallel L \parallel)$
defined above becomes the more familiar multiplier ideal $\JJ\big( ( X, \Lambda); \parallel L \parallel \big)$, and the exact sequence in Proposition \ref{basic_sequence_2} is simply
$$0 \longrightarrow \JJ ( (X, \Lambda); \parallel L \parallel) \otimes \OO_X(-S) \longrightarrow 
\Adj_S ((X, \Lambda);  \parallel L \parallel) \longrightarrow 
\JJ ( (S, \Lambda_S); \parallel L \parallel_{|S})\longrightarrow 0.$$
\end{remark}

We conclude this section with a technical statement on containments of ideals, which will be used later.

\begin{lemma}\label{klt_inclusion}
Let $X$ be a smooth variety. Let $\frb$ be an ideal sheaf and $\fra_{\bullet}$ a graded system of ideals on $X$.
 
\noindent
(1) If $\Lambda$ is an effective $\QQ$-divisor on $X$ such that the pair $(X, \Lambda)$ is klt, then 
$$\frb \subseteq  \JJ ( (X, \Lambda); \frb) {\rm ~and~} \fra_1 \subseteq  \JJ ( (X, \Lambda); 
\fra_{\bullet} ).$$

\noindent 
(2) If $D$ is a smooth divisor on $X$ such that the asymptotic adjoint ideals 
$\Adj_D  (X, \frb)$ and $\Adj_D  (X, \fra_{\bullet})$ are defined, then 
$$\frb \subseteq \Adj_D  (X, \frb) {\rm~and~} \fra_1 \subseteq 
\Adj_D (X, \fra_{\bullet}).$$
\end{lemma}
\begin{proof}
(1) The first part is \cite{takayama} Example 2.2(2), so let us prove the second one. By definition we have that $ \JJ ( (X, \Lambda); \fra_{\bullet}) =  \JJ ( (X, \Lambda); \fra_m^{1/m})$ for $m>>0$.
Fix such an $m$, and consider a common log-resolution $f:Y \rightarrow X$ for $\Lambda$, $\fra_1$ and $\fra_m$. If we write $\fra_1\cdot \OO_Y = \OO_Y (- E_1)$ and $\fra_m \cdot \OO_Y= \OO_Y (- E_m)$, then as 
$\fra_{\bullet}$ is a graded system we have $\frac{1}{m} E_m \subseteq E_1$. On the other hand, since 
$(X, \Lambda)$ is klt, the divisor $K_{Y/X} - [f^*\Lambda]$ is exceptional and effective. This gives the inclusions 
$$\OO_Y(-E_1) \subseteq \OO_Y (K_{Y/X} - [f^*\Lambda] -  E_1) \subseteq 
\OO_Y (K_{Y/X} - [f^*\Lambda + \frac{1}{m} E_m]),$$
which by push-forward imply what we want.

\noindent
(2) We show only the first inclusion. The second one follows similarly, as in (1). Consider $f: Y \rightarrow X$ a log-resolution for $\frb$, and write $\frb\cdot \OO_Y = \OO_Y(-E)$. The condition that 
$D$ is smooth implies that $K_{Y/X} - f^*D + \tilde D$ is exceptional and effective. (This is the triviality
of the usual adjoint ideal for a normal divisor with canonical singularities.) We obtain the inclusion
$$\OO_Y(-E) \subseteq \OO_Y(K_{Y/X} - f^*D  + \tilde D - E),$$
which again by push-forward implies the result.
\end{proof}

\section{Basic lifting}

In this section we state a general result about lifting sections which vanish along appropriate adjoint ideals. We give two versions that will be used in the text, one for adjoint ideals associated to divisors, and one for asymptotic adjoint ideals associated to linear series. There are of course similar statements in all the other situations discussed above. 

\begin{proposition}[Basic Lifting I]\label{basic_lifting_1}
Let $X$ be a smooth projective variety and $S\subset X$ a smooth divisor. Let $\Lambda$ and $B$ be effective $\QQ$-divisors on $X$, and $\Gamma$ an integral divisor such that $S + \Gamma$ is a reduced SNC divisor. 
Let $L$ be a line bundle on $X$ such that $L - \Lambda - B$ is big and nef, and assume that no log-canonical center of $(X, S+\Gamma)$ is contained in 
${\rm Supp}(\Lambda + B) \cup \BB_{+} (L- \Lambda - B)$. Then the sections in
$$H^0 (S, \OO_S(K_S + L_S + \Gamma_S) \otimes \Adj_{\Gamma_S} ((S, \Lambda_S); B_S) )$$ 
are in the image of the restriction map
$$H^0(X,  \OO_X(K_X + S+ L + \Gamma) )\rightarrow H^0(S,  \OO_S (K_S + L_S + \Gamma_S) ).$$
\end{proposition}
\begin{proof}
We use the exact sequence given by the analogue of Proposition \ref{basic_sequence_1}:
$$0\longrightarrow \Adj_{\Gamma}  ((X, \Lambda); B) \otimes \OO_X(-S) \longrightarrow \Adj_{\Gamma+ S} ((X, \Lambda); 
B)\longrightarrow $$
$$\longrightarrow \Adj_{\Gamma_S} ((S, \Lambda_S); B_S)\longrightarrow 0.$$ 
Twisting this sequence by $\OO_X(K_X + S+ L + \Gamma)$, the result follows if we show the 
surjectivity of the induced map on the right hand side at the level of $H^0$. But this is in turn implied by the vanishing
$$H^1 (X, \OO_X(K_X + L + \Gamma ) \otimes \Adj_{\Gamma} ((X, \Lambda); B)) = 0,$$
which is a consequence of Theorem  \ref{adjoint_nadel}.
\end{proof}

\begin{proposition}[Basic Lifting II]\label{basic_lifting_2}
Let $X$ be a smooth projective variety and $S\subset X$ a smooth divisor. Let $\Gamma$ be an effective integral divisor on $X$, such that $S + \Gamma$ is a reduced SNC divisor. Consider $\Lambda$ an 
effective $\QQ$-divisor, and $L$ a big line bundle on $X$, such that no log-canonical center of $(X,S + \Gamma)$ is contained in $\BB(L) \cup {\rm Supp }(\Lambda)$ and no log-canonical center of 
$(X, \Gamma)$ is contained in $\BB_{+} (L)$. If $A$ is an integral divisor on $X$ such that $A - \Lambda$ is nef, then the sections in 
$$H^0 (S, \OO_S (K_S + A_S + \Gamma_S + L_S) \otimes \Adj_{\Gamma_S} ((S, \Lambda_S); \parallel L \parallel_{|S}))$$ 
are in the image of the restriction map 
$$H^0 (X, \OO_X (K_X + S +  A + \Gamma + L))\longrightarrow  
H^0 (S, \OO_S (K_S +  A_S + \Gamma_S + L_S)).$$
\end{proposition}
\begin{proof}
The hypotheses imply that we can run the short exact sequence in Proposition \ref{basic_sequence_2}. 
After twisting that sequence by $\OO_X(K_X + S +  A + \Gamma + L)$, the statement follows as above from the vanishing
$$H^1 (X, \OO_X(K_X + A + \Gamma + L)\otimes \Adj_{\Gamma} ((X, \Lambda); \parallel L \parallel)) = 0,$$
which is a consequence of Theorem \ref{asymptotic_nadel}.
\end{proof}

\section{Extension theorems}

We start by addressing a small technical point. The results below are about lifting sections of 
line bundles $\OO_S(m L_S)$ (perhaps embedded in larger global sections groups), where $S \subset X$ is a smooth divisor, and $L$ is a line bundle on $X$. If there are no such section for any $m$, then the results are vacuous. Otherwise we have $\kappa (L_S) \ge 0$, and also $\kappa(L) \ge 0$ once we have lifting at any stage, hence all asymptotic multiplier ideals are well-defined.

\subsection*{A proof of Takayama's extension theorem.}
We start with a proof of Takayama's extension result. The reason for including it here is that the argument for the first step differs from the one in \cite{takayama}, and is intended to be a less technical toy version of the proof we will give for the more general statement Theorem \ref{extension0}. 

\begin{theorem}[\cite{takayama}, Theorem 4.1]\label{takayama}
Let $X$ be a smooth projective variety and $S$ a smooth divisor in $X$. Let $L$ be a Cartier divisor on $X$ such that  $L \sim_{\QQ}  A + \Delta$, where $\Delta$ is an effective $\QQ$-divisor such that $S \not\subset {\rm Supp}~\Delta$, with the pair $(S, \Delta_S)$ klt, and $A$ is a nef and big $\QQ$-divisor such that $S \not\subseteq \BB_{+} (A)$. Then for any $m\ge 1$, the restriction map 
$$H^0 (X, \OO_X (m (K_X + S + L))) \longrightarrow H^0(S, \OO_S (m(K_S + L_S)))$$ 
is surjective. 
\end{theorem}
\begin{proof}
Fix a sufficiently positive ample divisor $H$ on $X$, to be specified later, and denote by $H_S$ its restriction to $S$, given by $h_S = 0$. 

\noindent
\emph{Step 1.} Denote $M: = K_X + L + S$, and consider $\frc_m : = b_{|mM + H|}\cdot \OO_S$. This is the base-ideal of the \emph{restricted} linear series, i.e. the image of the restriction map
$$H^0 (X, \OO_X (mM + H)) \longrightarrow H^0 (S, \OO_S (mM_S + H_S)).$$

\noindent
\emph{Claim}: For any $m\ge 0$, the image of 
$$H^0 (S, \OO_S(mM_S)) \overset{\cdot h_S}{\longrightarrow} H^0 (S, \OO_S(mM_S + H_S))$$
can be lifted to $X$.  Equivalently, the sections in  
$H^0 (S, \OO_S(mM_S+ H_S)\otimes \JJ((S, \Delta_S); \parallel mM_S \parallel))$
can be lifted to sections in $H^0 (X, \OO_X(mM+ H))$.

We prove this by induction on $m$. We can take $H$ sufficiently positive so that the first few steps are satisfied 
by Serre Vanishing. Let's assume now that the statement is known up to some integer $m$, and prove it for 
$m+1$. First note that we can choose $H$ so that the Claim at level $m$ implies (so is in fact equivalent to)
$$\JJ \big( (S, \Delta_S); \parallel mM_S\parallel \big) \subseteq \frc_m.$$ 
Indeed, by the definition of $\frc_m$ this follows if $\OO_S(mM_S + H_S)\otimes \JJ \big( (S, \Delta_S); \parallel mM_S\parallel \big)$ is generated by global sections. But this follows say by choosing $H$ of the form 
$H = K_X + S + (n+1) B + C$, where $n$ is the dimension of $X$, $B$ a sufficiently positive very ample line bundle on $X$, and $C$ any other ample line bundle, by a standard application of Nadel Vanishing and Castelnuovo-Mumford regularity (cf. \cite{positivity}, Corollary 11.2.13).
Note on the other hand that by Lemma \ref{klt_inclusion}(1) we have   
$$\frc_m \subseteq \JJ \big( (S, \Delta_S); \parallel mM + H \parallel_{|S} \big).$$
We can equally assume, by the same application of Nadel Vanishing and Castelnuovo-Mumford regularity, that  $H$ can be chosen sufficiently positive, but independent of $m$, so that 
$$\OO_S((m+1)M_S + H_S) \otimes   \JJ \big( (S, \Delta_S); \parallel mM_S\parallel \big)$$
is generated by global sections. (Note that $(m+1)M_S + H_S = K_S + mM_S + L_S + H_S$.)
Therefore have a sequence of inclusions
$$H^0(S,  \OO_S((m+1)M_S + H_S)
\otimes   \JJ \big( (S, \Delta_S); \parallel mM_S\parallel \big))
\subseteq$$
$$\subseteq H^0(S,  \OO_S((m+1)M_S + H_S) \otimes  \frc_m) \subseteq$$
$$\subseteq H^0 (S,  \OO_S((m+1)M_S + H_S) \otimes \JJ \big( (S, \Delta_S); \parallel mM + H \parallel_{|S} \big) ).$$
Since in any case 
$$\JJ \big( (S, \Delta_S); \parallel (m+1)M_S\parallel \big) \subseteq 
\JJ \big( (S, \Delta_S); \parallel mM_S\parallel \big),$$
the inductive step follows if we show that the sections of $\OO_S((m+1)M_S + H_S)$ vanishing along the ideal $\JJ \big( (S, \Delta_S); \parallel mM + H \parallel_{|S} \big)$ lift to sections of $\OO_X((m+1)M + H)$.
But this follows by Basic Lifting, Proposition \ref{basic_lifting_2} (and Remark \ref{no_restriction}), provided that the adjoint ideal $\Adj_S ((X, \Delta); \parallel mM + H \parallel)$ can be defined, in other words provided that $S \not
\subset \BB (mM + H)$. 

This last things holds of course, in a strong sense, if we show that $H$ can be conveniently chosen so that 
there exist non-zero sections of $\OO_S(mM_S + H_S)$ that lift to $X$. The problem with applying the inductive step 
directly is that $mM_S$ itself might not have any sections for our given $m$. This is circumvented as follows: 
note that there exists an integer $k_0$ such that $\BB(M_S) = {\rm Bs}(pk_0M_S)$ for all $p$ sufficiently large, say $p \ge p_0.$\footnote{The existence of such a $k_0$ and the fact that it may be strictly greater than $1$ are standard -- cf. \cite{positivity} Proposition 2.1.21.} In addition to the conditions already imposed, we require $H$ to satisfy the following:
for all $1\le q < p_0k_0$, the divisor $H_q : = q M + H$ is again ample, and it satisfies all the generation properties required of $H$ itself (e.g. by absorbing all $q M$ in the ample divisor $C$ mentioned above). Now write $m = pk_0 + q$, with $p$ divisible by $p_0$, and $0\le q < p_0k_0$. Then we deduce that $\OO_S(mM_S + H_S)$ has sections that lift to $X$, by virtue of the fact that we can rewrite 
$$mM_S + H_S = p k_0 M_S + H_{q, S},$$
and the non-trivial sections in $H^0 (S, \OO_S(p k_0 M_S))$ can be lifted to $X$ by induction, after being multiplied by an equation of $H_{q, S}$.

\begin{remark}
Note that in this step it would have been enough to assume that $A$ is only nef. This is the origin of 
Remark \ref{simplification} in the Introduction. For Step 2 it is however necessary to assume that $A$ 
is also big.
\end{remark}

\noindent
\emph{Step 2.} In this step one gets rid of $H$. This part follows Takayama's proof identically. We present this in Step 5 of the proof of Theorem \ref{extension0}  given below, so we skip it here.  
\end{proof}

\subsection*{An example.}
Takayama's theorem shows that sections of adjoint line bundles of the form $K_S + L_S$, where $L$ is an ample line bundle on $X$, or log-versions of these, can be extended without any a priori assumption on the line bundle  $K_X + S + L$ on the ambient space. Siu's results (cf. also Theorem \ref{fibration}) say that the same is 
true for line bundles of the form $p K_{X_0} + L_{X_0}$ for any $p\ge 1$, in the case of a smooth 
fibration over a curve, with central fiber $X_0$.
Here we give an example showing that in the case of restriction to a divisor $S$ with nontrivial 
normal bundle, sections of line bundles of the form $p K_S + L_S$ do not necessarily extend 
for $p \ge 2$. Thus in the general case it is required to assume from the beginning that 
$p K_X + L $ be at least pseudo-effective for extension to work.

Let $\pi : X \rightarrow C$ be the ruled surface $X = \PP ( \OO_C \oplus \OO_C (1))$ over a smooth projective curve $C$ of genus $g \ge 2$. Denote by $S$ the section of $\pi$ corresponding to the $\OO_C$ factor, by $H$ a divisor 
on $C$ corresponding to $\OO_C(1)$, and by $D$ the section corresponding to the $\OO_C(1)$ factor, belonging 
to the linear system $|\OO_X(1)|$. Then we have the following linear equivalences (cf. \cite{hartshorne} Proposition 2.9  and Lemma 2.10):
$$S \sim D - \pi^* H {\rm ~and ~} K_X \sim -2D + \pi^* K_C + \pi^* H.$$ 
This gives $K_X + S \sim - D + \pi^* K_C$. This in turn implies that no multiple of $K_X + S + \epsilon A$ can 
have any sections, for any ample Cartier divisor $A$ and $\epsilon$ sufficiently small. Indeed, the intersection 
number with a fiber $F$ is 
$$ (- D + \pi^* K_C + \epsilon A) \cdot F = - 1 + \epsilon A \cdot F, $$
which is negative for $0 < \epsilon \ll 1$. 
On the other hand $K_X + S|_{S} \sim K_S$, which is ample, so multiples of $K_S + \epsilon A_S$ have lots of 
sections for any $\epsilon$.
One can easily replace in this example the curve $C$ by any variety of general type.

\subsection*{The general statement.} 
For the sake of simplicity during the proof, we introduce the following ad-hoc terminology: given a 
pair $(X, \Delta)$, the irreducible closed subsets $W \subset X$ such that ${\rm mld}(\mu_W; X, \Delta) <1$ will be called \emph{pseudo non-canonical centers} of $(X, \Delta)$ (note that those which are centers corresponding to exceptional divisors are indeed non-canonical centers of the pair).
We also make use of the following Lemma, left to the reader as it can be easily checked from the definitions in the Introduction. (Note that both inclusions are in fact equalities, but this is harder to check, especially in the case of $\BB_{-}$.)

\begin{lemma}\label{inclusions}
Let $f: Y \rightarrow X$ be a projective birational morphism of normal varieties, and let $D$ be a 
$\QQ$-divisor on $X$. Then $\BB_{+}(f^* D) \subseteq f^{-1}(\BB_{+} (D)) \cup {\rm Exc}(f)$ and 
$\BB_{-} (f^*D) \subseteq f^{-1}(\BB_{-} (D))$.
\end{lemma}

As the proof of Theorem \ref{extension0} is quite technical, to help the reader navigate through it we start with a general outline of the steps involved. 

\noindent
{\bf Outline.}  In Steps 1-5 we will prove the theorem under the stronger assumption that $M$ is $\QQ$-effective, and the pseudo non-canonical centers of $(S, \Delta_S)$ are not contained in the \emph{stable} base locus $\BB(M_S)$. We will show how to deduce the statement involving only the restricted base locus $\BB_{-} (M_S)$, for a pseudo-effective $M$, in Step 6.

\noindent
\emph{Step 1.}
We reduce to the case when $X$ is smooth and the divisor $S + \Delta$ has 
SNC support via a discrepancy  calculation and basic facts about restricted and augmented base loci.

\noindent
\emph{Step 2.}
Similarly to Step 1, using special log-resolutions we further reduce to the case when all the pseudo non-canonical centers of $(S, \Delta_S)$ are disjoint irreducible divisors on $S$, and all the pseudo non-canonical centers of $(X, \Delta)$ disjoint from $S$ are disjoint irreducible divisors as well.

\noindent
\emph{Step 3.}
We introduce the statement of the main technical step towards full extension, Proposition 
\ref{intermediate}. This provides the lifting of sections vanishing along certain asymptotic multiplier ideals 
after we add a sufficiently positive fixed ample line bundle. We then perform another 
reduction, to the case when there are no irreducible components of 
${\rm Supp}(\Delta)$ disjoint from $S$ which are contained in the stable base locus of  $M$.
  
\noindent
\emph{Step 4.}
This is the main step, proving the result introduced in Step 3, i.e. extension after 
adding a fixed positive quantity. Fundamentally the idea is similar to the proof of Takayama's theorem earlier in this section, relying on Basic Lifting for asymptotic multiplier and adjoint ideals. The main technical difference is that our line bundle is now of the form $M = k(K_X + S + A + \Delta)$, 
where $k \ge 2$. To deal with this, the inductive step of going from $mM + H$ to $(m+1) M + H$ is subdivided into $k$ sub-steps as in Hacon-M$^c$Kernan \cite{hm1}. Extension at the various stages is 
made formal by the use of the inductive sequences for studying adjoint ideals, as in 
Proposition \ref{basic_sequence_2}.

\noindent
\emph{Step 5.}
We get rid of the ample line bundle introduced in Step 3, following an argument of Takayama
(which is another consequence of Basic Lifting). This is the first time, after performing the reductions in the first two steps, when we use the fact that $A$ is big.

\noindent
\emph{Step 6.}
We show how to deduce from the above the statement involving only the restricted base locus 
$\BB_{-} (M_S)$, for a pseudo-effective $M$, based on the argument in Corollary \ref{extension2} together with another use of Takayama's reduction as in Step 5.

\begin{proof}(of Theorem \ref{extension0}.)
\emph{Step 1.}
We first reduce to the case when $X$ is smooth and the divisor $S + \Delta$ has 
SNC support.
Consider a log resolution $f: Y \rightarrow X$ of the pair $(X, S+ \Delta)$. Denote by $\tilde S$
and $\tilde \Delta$ the corresponding proper transforms. We write 
$$K_Y + \tilde S + \tilde \Delta + f^* A = f^* (K_X  + S + A + \Delta) + P - N,$$
where $P$ and $N$ are exceptional effective $\QQ$-divisors with no common components. 
Note that $kP$ and $kN$ are integral divisors, by the choice of $k$. We can consider a Cartier divisor 
on $Y$
$$\tilde M := k(K_Y + \tilde S + \tilde \Delta + N + f^*A) = f^* k (K_X  + S + \Delta + A) + kP.
\footnote{Note here that if $X$ is smooth, we can assume only that $M\sim_{\QQ} k(K_X + S + \Delta + A)$, and 
consequently $\tilde M \sim_{\QQ} k(K_Y + \tilde S + \tilde \Delta + N + f^*A)$, as this is enough to make 
$kP$ and $kN$ Cartier. Same in Step 2 below, and this leads to the final sentence in the statement of Theorem \ref{extension0}.}$$
Since $kP$ is effective and exceptional,  the sections
of $\OO_Y( m \tilde M)$ are the same as those of $\OO_X (mM)$ for all $m$. 
Since the sections $H^0 (S, \OO_S (mM_S))$ inject into the sections $H^0 (\tilde S, \OO_{\tilde S} 
(m {\tilde M}_{\tilde S}))$,  it follows that it suffices to prove the surjectivity
$$H^0 (Y, \OO_Y (m\tilde M)) \longrightarrow H^0 (\tilde S, \OO_{\tilde S} (m{\tilde M}_{\tilde S}))$$
for all $m$. It is then enough to show that we can replace $\Delta$ with $\tilde \Delta + N$ on $Y$. 
To this end, note to begin with that since $(X, S + \Delta)$ is plt it follows that 
$[\tilde \Delta + N] = 0$. As everything on $Y$ is in simple normal crossings, all the singularity hypotheses required
for $(Y, \tilde S + \tilde \Delta + N)$ are obviously satisfied. Moreover, by Lemma \ref{inclusions} we have that 
$\BB_{+} (f^* A) \subseteq f^{-1} (\BB_{+}(A)) \cup {\rm Exc}(f)$, so $\tilde S \not \subseteq \BB_{+}(f^*A)$.

Next we show that $\BB ({\tilde M}_{\tilde S})  = f^{-1} (\BB (M_S)) \cup {\rm Supp} (P_{\tilde S})$\ does not contain any of the pseudo non-canonical centers of $(\tilde S, {\tilde \Delta}_{\tilde S} + N_{\tilde S})$. This is a discrepancy calculation. Note that by construction and the SNC hypothesis, $P_{\tilde S}$ cannot contain any of these pseudo non-canonical centers. It is then enough to show that the image of any such pseudo non-canonical center via $f$ is equal to a pseudo non-canonical center of
the pair $(S, \Delta_S)$. 

Every pseudo non-canonical center of the SNC pair $(Y, \tilde \Delta + N)$ is dominated by a divisor with log-discrepancy less than $1$ on a further blow-up, so by going to a higher model it is enough to concentrate on irreducible components of 
$\lceil \tilde \Delta + N \rceil$ intersecting $\tilde S$. The components of $\lceil \tilde \Delta_{\tilde S} \rceil$ automatically satisfy our property 
by the assumption on $\Delta_S$. Hence it remains to show that the components of the support of $N$ intersected with $\tilde S$ map to pseudo non-canonical centers of  $(S, \Delta_S)$. Let $\Gamma \subset {\rm Supp}(N)$ be such an irreducible component intersecting $\tilde S$. We have  
$$\ord_{\Gamma}(N) = \ord_{\Gamma} (- K_Y + f^* (K_X + S+ \Delta)).$$
Since the pair $(X, S+ \Delta)$ is plt, we have that $\ord_{\Gamma}(\lceil N\rceil) = 1$.
On the other hand, by adjunction $ (K_Y - f^* (K_X + S))_{|\tilde S} = f_{|\tilde S}^* K_S$ , and since everything is in normal crossings we obtain 
$$\ord_{\Gamma} (N) = \ord_{\Gamma_{\tilde S}}(N_{\tilde S}) = \ord_{\Gamma_{\tilde S}} (- K_{\tilde S} + f_{|\tilde S}^* (K_S + \Delta_S)).$$
This implies that $f$ maps $\Gamma_{\tilde S}$ onto a pseudo non-canonical center of $(S, \Delta_S)$.

It remains to check that $\BB_{-} (\tilde M)$, which again by Lemma \ref{inclusions} plus the fact that 
$P$ is exceptional, is contained in $f^{-1} (\BB_{-} (M)) \cup {\rm Supp} (P)$, does not 
contain any pseudo non-canonical center of $(Y, {\tilde \Delta} + N)$. This follows by an argument completely analogous to that in the previous paragraph.

\noindent
\emph{Step 2.}
Assuming that $X$ is smooth and $S+ \Delta$ has SNC support, 
we now show that we can further reduce to the case when all the pseudo non-canonical centers of $(S, \Delta_S)$ 
are disjoint irreducible divisors on $S$, and at the same time 
all the pseudo non-canonical centers of $(X, \Delta)$ disjoint from $S$ are disjoint irreducible divisors as well.

Note first that $(X, \Delta)$ and $(S, \Delta_S)$ are klt pairs with SNC support. 
We can apply twice Corollary \ref{separation} below. First we apply it on $S$ to make 
the pseudo non-canonical centers of $(S, \Delta_S)$ simply a union of disjoint irreducible divisors (i.e. the 
corresponding components of $\Delta$ intersect only outside $S$). Then we apply it again for $(U, \Delta_U)$,
where $U = X - S$, to have all the pseudo non-canonical centers of $(X, \Delta)$ outside of $S$ also a union of disjoint irreducible divisors.
We obtain a birational model $f: Y \rightarrow X$ on which we write 
$$K_Y + \tilde S + \tilde \Delta + f^*A = f^* (K_X  + S + \Delta + A) + P - N,$$
where $P$ and $N$ are exceptional effective $\QQ$-divisors with no common components. 
The divisor $\tilde \Delta + N$ replaces $\Delta$, and has all the properties stated above for 
its pseudo non-canonical centers: indeed, by our choice of resolution, the pseudo non-canonical centers of 
$(Y, \tilde \Delta + N)$ are precisely the irreducible components of the support of 
$\tilde \Delta + N$, of which those that do not intersect $S$ are disjoint. The reduction to working on $Y$ is 
done precisely as in Step 1, and we do not repeat it.

\noindent
\emph{Step 3.} 
The reductions in Steps 1 and 2 being made, we appeal to the usual technique of introducing extra positivity as an intermediate step, to prove the following statement:

\begin{proposition}\label{intermediate}
Let $X$ be a smooth projective variety, and $S\subset X$ a smooth irreducible divisor. Let $\Delta$ be an effective $\QQ$-divisor on $X$ such that $[\Delta] = 0$ and $S \not \subset {\rm Supp}(\Delta)$. Let $B$ be a nef $\QQ$-divisor, $k$ a positive integer and $M$ a Cartier divisor such that $M \sim_{\QQ} k(K_X + S + B + \Delta)$. Assume the following:
\begin{itemize}
\item $S + \lceil \Delta \rceil$ is a simple normal crossings divisor. 
\item $M$ is $\QQ$-effective.
\item the restricted base locus $\BB_{-} (M)$ does not contain any pseudo non-canonical centers of $(X, \Delta)$ which intersect $S$.
\item the stable base locus  $\BB (M_S)$ does not contain any pseudo non-canonical centers of 
$(S, \Delta_S)$.
\item the only pseudo non-canonical centers of $(X, \Delta)$ disjoint from $S$ are irreducible components of ${\rm Supp}(\Delta)$.
\end{itemize}
Then for a sufficiently ample divisor $H$ the sections in  
$$H^0 (S, \OO_S(mM_S+ H_S)\otimes \JJ(S, \parallel mM_S \parallel))$$
can be lifted to sections in $H^0 (X, \OO_X(mM+ H))$, for all sufficiently divisible $m\ge 0$. 
\end{proposition}

\noindent
Note that we always have an isomorphism 
$$H^0 (S, \OO_S(mM_S) ) =  
H^0 (S, \OO_S(mM_S) \otimes \JJ (S, \parallel mM_S\parallel ) ).$$ 
This means that the statement can be interpreted as saying that if we fix any section $h_S \in 
H^0(S, H_S)$, then the image of the inclusion 
$$H^0 (S, \OO_S(mM_S) ) \overset{\cdot h_S}{\longrightarrow} H^0 (S, \OO_S(mM_S + H_S) ) $$ 
can be lifted to $X$. In the proof below, we will interpret the condition of $m$ being sufficiently divisible 
as potentially replacing $k$ by some conveniently chosen multiple.

To prove Proposition \ref{intermediate}, we first make another reduction. We are already assuming that
the only pseudo non-canonical centers of $(X, \Delta)$ disjoint from $S$ are the irreducible components of 
$ {\rm Supp}(\Delta)$ disjoint from $S$ (so they do not intersect outside 
$S$), and we show that we can further reduce to the case when there are no such components
whatsoever which are contained in $\BB_{-} (M)$.

Say $\Gamma$ is a component of the support of $\Delta$ disjoint from $S$. We denote by $\gamma$ the coefficient of 
$\Gamma$ in $\Delta$. If $\Gamma \subset \BB (M)$, consider the asymptotic order of vanishing of $K_X + S + \Delta + A$ along $\Gamma$:
$$c:= \ord_{\Gamma} \parallel K_X + S + \Delta + A \parallel =  \ord_{\Gamma} \parallel \frac{M}{k} \parallel .$$
We know (cf. \cite{elmnp} \S2) that 
$$c = \underset{p\to \infty}{\rm lim}  \frac{\ord_{\Gamma} |p\frac{M}{k}|}{p} = 
\underset{p}{\rm inf} \frac{\ord_{\Gamma} |p\frac{M}{k}|}{p},$$
where for a linear series, the order of vanishing along a divisor is by definition that of a general member of the linear series.
This may a priori be a real number, but for $p$ sufficiently large and divisible we have that $c_p := \frac{\ord_{\Gamma} |p\frac{M}{k}|}{p} \ge c$ is a rational number arbitrarily close to $c$. We consider different cases, according to how 
$\gamma$ compares to $c$.

If $c \ge \gamma$, we can replace in our argument $\Delta$ by 
$\Delta - \gamma \Gamma$ and correspondingly $M$ by $M - k \gamma \Gamma$. All our hypotheses are still 
satisfied with this new choice, only we have to essentially restrict to integers $k$ which make $k\Delta$ Cartier. (Note that by the definition of $c$ we have that $M - k \gamma \Gamma$ is still $\QQ$-effective.) Since we are subtracting something effective, it is enough to prove extension from $S$ 
for the multiples of this new divisor, since then we can simply add back the corresponding multiple of $k \gamma \Gamma$. 
With this new choice we do not have to worry about the component $\Gamma$ being contained in $\BB (M)$ any more.

If $\gamma > c$, then we have $\gamma > c_p$ for $p$ large enough. We can then replace in our argument $\Delta$ by 
$\Delta - c_p \Gamma$ and correspondingly $M$ by $M - k c_p \Gamma$. Indeed, all the hypotheses are obviously preserved for this new choice, only this time we also have to assume that $k$ is such that $kc_p$ is an integer. (Note that $pk c_p \Gamma$ is contained in the base locus of $|pM|$, so this new divisor 
is also $\QQ$-effective.) Since we are subtracting something effective, it is enough to prove extension from $S$ 
for the multiples of this new divisor, since then we can simply add back the corresponding multiple of $k c_p \Gamma$. 
In addition,  by the definition of $c$, this time we have $\Gamma \not\subset \BB (M - k c_p \Gamma)$, so by this process 
we are able to reduce to the case when $\Gamma$ is not contained in $\BB (M)$.

Doing this for every component $\Gamma$ as above, we see that by suitably modifying the divisor $\Delta$, we can assume that there are no pseudo non-canonical centers of $(X, \Delta)$ disjoint from $S$ which are contained in 
$\BB_{-} (M)$ (which is contained in $\BB (M)$).

\noindent
\emph{Step 4.} According to the previous step, we continue the proof of Proposition \ref{intermediate} assuming in addition that there are no pseudo non-canonical centers (which by now are components) of 
${\rm Supp}(\Delta)$ disjoint from $S$ which are contained in $\BB_{-} (M)$. Since every non-canonical center involved is now an irreducible component of some reduced divisor, we will switch to the language of log-canonical centers, used in the definition of adjoint ideals.
To conclude the proof, in this case it suffices to assume that $M$ is only pseudo-effective, or equivalently $M + A$ is big for every ample $\QQ$-divisor $A$. 

Let us also note that $H$ can be chosen sufficiently positive so that $H^0 (S, \OO_S (mM_S + H_S))\neq 0$ for all $m \ge 0$. Indeed, we can assume that there is an $m_0$ such that $H^0 (S, \OO_S (m_0 M_S))\neq 0$ (otherwise the result 
is trivial), and we can choose $H$ so that $H^0 (S, \OO_S (rM_S + H_S))\neq 0$ for all $0 \le r \le m_0 -1$. Now simply divide any $m$ by $m_0$, with remainder. 

We first rewrite the divisor $M$ in a more convenient form (writing equality instead of linear equivalence where it makes no difference). Denote $\Delta_0 := \{k\Delta\}$, so that $k\Delta = \Delta_0 + [k\Delta]$. By the choice of $k$, we have that $L := kB + \Delta_0$ is an integral divisor. By assumption the pair $(S, \Delta_{0,S})$ is klt. Since  all the coefficients of the components of $[k\Delta]$ are at most $k-1$,  we can also write 
$$[k\Delta] = \Delta_1 + \ldots + \Delta_{k-1}$$ 
where the $\Delta_i$'s are all reduced SNC divisors, with supports increasing with $i$. This means we can rewrite $M$ as 
$$M = k(K_X + S) + L + \Delta_1 + \ldots + \Delta_{k-1}.$$

Consider $\frc_m : = b_{|mM + H|}\cdot \OO_S$, i.e. the base-ideal of the linear series
given by the image of the restriction map
$$H^0 (X, \OO_X (mM + H)) \longrightarrow H^0 (S, \OO_S (mM_S + H_S)).$$
Let's remark that in the first place $H$ can be chosen sufficiently positive so that for all $m$ the sheaf 
$$\OO_S(mM_S + H_S) \otimes \JJ(S, \parallel mM_S \parallel )$$
is globally generated (again by the usual Nadel Vanishing and Castelnuovo-Mumford regularity argument -- cf. \cite{positivity}, Corollary 11.2.13). The conclusion of the Theorem for a given $m$ implies then in particular that 
\begin{equation}\label{step0}
\JJ\big(S,  \parallel mM_S\parallel \big) \subseteq \frc_m.
\end{equation}

We will prove Theorem by induction on $m$. It is clear for $m = 0$, and we show that it holds for $(m+1)$
assuming that  ($\ref{step0}$) holds for $m$. Assuming that everything is defined, we prove in fact the second inclusion in the following sequence (the first one is obvious):
\begin{equation}\label{main}
\JJ\big(S, \parallel (m+ 1) M_S\parallel \big) \subseteq 
\JJ\big(S, \parallel mM_S\parallel \big) \subseteq 
\end{equation}
$$\subseteq 
\JJ \big( (S, \Delta_{0,S}); \parallel mM + H  + (k-1)(K_X + S)  + \Delta_1 + \ldots  + \Delta_{k-1}\parallel _{|S} \big).$$
Let's first see that this concludes the induction step. We only need to show that the sections of this last multiplier ideal, twisted by $\OO_S ((m+1)M_S + H_S)$, can be lifted to $X$. But this follows from Basic Lifting. Indeed, note that 
$$(m+1) M + H - S = K_X + N + kB + \Delta_0,$$
where $N : = mM + H + (k-1) (K_X + S) + \Delta_1 + \ldots  + \Delta_{k-1}$.
Since $B$ is nef, Basic Lifting given by Proposition \ref{basic_lifting_2} and Remark \ref{no_restriction} applies if we have the 
following two conditions: $N$ is big, and $S \not\subset \BB (N)$. But since $M$ is pseudo-effective, $H$ can certainly be chosen 
sufficiently positive so that $N$ is big (independently of $m$), while the second condition will follow inductively from the proof.  

We are left then with proving the second inclusion in ($\ref{main}$). We do this in several steps. 
The first step is to note that
\begin{equation}\label{step1}
\JJ\big(S, \parallel mM_S\parallel \big) \subseteq 
\Adj_{\Delta_{1,S}}  (S,  \parallel mM + H \parallel_{|S} ). 
\end{equation}
This follows since on  one hand we are assuming that 
$\JJ\big(S, \parallel mM_S\parallel \big) \subseteq \frc_m$, while on the other hand by
Lemma \ref{klt_inclusion}(2)
$$\frc_m \subseteq \Adj_{\Delta_{1,S}}  (S,  \parallel mM + H \parallel_{|S} ),$$
as long as the adjoint ideal involved is defined (cf. \S3). 
To see this, we will check the stronger statement that there exist sections of 
$mM_S + H_S$ which do not vanish along any intersection of a component of ${\rm Supp}([k\Delta])$ with $S$ (by the assumption in this step, these are the only log-canonical centers of 
${\rm Supp}([k\Delta])_S$), 
and lift to $X$. By induction we know that we can lift the sections of $mM_S$ after adding $H_S$. This space could a priori be empty for some values of $m$, but since we are allowed here to work with sufficiently divisible integers, let us assume from the beginning that $k$ is chosen such that $\BB(M_S) = {\rm Bs}(p M_S)$ for all $p \ge 1$. Thus there are such nonzero sections which lift, and by the assumption on $\BB(M_S)$ they  do not vanish along the required intersections.

Note for use in the next step that as a consequence the adjoint ideal $\Adj_{S + \Delta_1} (X,  \parallel mM + H \parallel \big)$ is defined as well. Indeed, the sections we just lifted from $S$ are sections of $mM+ H$ that do not vanish along the components of $S+ \Delta_1$ intersecting $S$. But by the assumption in this step, these could have been the only components of $S + \Delta_1$ contained in $\BB (mM + H) (\subseteq \BB (M))$. Note however, as a slightly subtle point 
to come up later, that this does not mean that the components of $\Delta_1$ which intersect $S$ are not contained in 
$\BB_{-} (M)$.\footnote{It can be easily seen that, for a $\QQ$-divisor $D$ and any fixed ample $\QQ$-divisor $H$, one has 
$$\BB_{-} (D) = \bigcup_{m\ge 1} \BB(D + \frac{1}{m} H) = \bigcup_{m\ge 1} \BB_{+} 
(D + \frac{1}{m} H).$$
Cf. \cite{elmnp} Proposition 1.19 and Remark 1.20.}

The next step is to show that 
\begin{equation}\label{step2}
\JJ\big(S, \parallel mM_S\parallel \big) \subseteq 
\Adj_{\Delta_{2,S}}  (S,  \parallel mM + H + K_X + S + \Delta_1\parallel_{|S}),
\end{equation}
including the fact that the ideal on the right is defined.
To this end, note that by the usual Castelnuovo-Mumford regularity argument mentioned above we can choose $H$ positive enough (independently of $m$) so that the sheaf 
$$\OO_S (K_S + mM_S + H_S + \Delta_{1,S}) \otimes 
\JJ\big(S, \parallel mM_S\parallel \big)$$
is globally generated. Assume that the following sequence of inclusions of spaces of global sections holds, and everything is well defined:
$$ H^0 (S, \OO_S (K_S + mM_S + H_S + \Delta_{1,S}) \otimes 
\JJ\big(S,  \parallel mM_S\parallel \big)) \subseteq $$
$$\subseteq H^0 (S, \OO_S (K_S + mM_S + H_S + \Delta_{1,S}) \otimes 
\Adj_{\Delta_{1,S}}  (S,  \parallel mM + H \parallel_{|S} ))\subseteq $$
$$\subseteq H^0 (S, \OO_S (K_S + mM_S + H_S + \Delta_{1,S}) \otimes \frc_{|mM + H + K_X + S + \Delta_1|}) \subseteq $$
$$\subseteq H^0 (S, \OO_S (K_S + mM_S + H_S + \Delta_{1,S}) \otimes 
\Adj_{\Delta_{2,S}}  (S,  \parallel mM + H + K_X + S + \Delta_1\parallel_{|S})).$$
Then given the first and last terms in the sequence, and the global generation above, ($\ref{step2}$) follows.
Here is the explanation for the sequence of inclusions. The first is a direct consequence of ($\ref{step1}$).
The second follows if we show that the sections in 
$$H^0 (S, \OO_S (K_S + mM_S + H_S + \Delta_{1,S}) \otimes 
\Adj_{\Delta_{1,S}} (S,   \parallel mM + H \parallel_{|S}))$$
can be lifted to sections of $\OO_X( mM + H  + K_X+ S + \Delta_1)$. But this is once more a consequence of Proposition 
\ref{basic_lifting_2}, which uses the fact that the ideal $\Adj_{S+ \Delta_1} (X,  \parallel mM + H \parallel)$ is defined, explained in the paragraph above. Note the fact, alluded to earlier, that
Proposition \ref{basic_lifting_2} (in fact Nadel Vanishing) can be applied, since by assumption the  
components of $\Delta_1$ which intersect $S$ are not contained in $\BB_{-} (M)$, hence also not 
in $\BB_{+} (mM + H)$. 
The third one follows again from the inclusion given by Lemma \ref{klt_inclusion}(2)
$$ \frc_{|mM + H + K_X + S + \Delta_1|}\subseteq 
\Adj_{\Delta_{2,S}}  (S,  \parallel mM + H + K_X + S + \Delta_1\parallel_{|S}).$$
Here of course we denote by  $\frc_{|mM + H + K_X + S + \Delta_1|}$ the base ideal of the restriction of the corresponding 
linear series to $S$. The explanation of why the adjoint ideal on the right is defined is completely similar to that in the case of $\Adj_{\Delta_{1,S}}  (S,  \parallel mM + H \parallel_{|S} )$. The only difference is that instead of $H$ we 
have to consider the divisor $H_1: = H + K_X + S + \Delta_1$. But certainly $H$ can be made sufficiently positive, independently of $m$, so that all of the properties required of $H$ at the first step are also satisfied by $H_1$,  hence it can be used in the inductive step as well. With this
choice, let us finally observe that the sections lifted to show the second inclusion above are in fact,
after adding the ample divisor $H_S + K_S + \Delta_{1,S}$,  
all the sections in $H^0 (S, \OO_S (mM_S))$, due to the inclusion in ($\ref{step1}$).

\noindent
By the same token, the next step is to show the inclusion 
\begin{equation}\label{step3}
\JJ\big(S, \parallel mM_S\parallel \big) \subseteq 
\Adj_{\Delta_{3,S}} (S,  \parallel mM + H + 2(K_X + S) + \Delta_1 + \Delta_2 \parallel_{|S}), 
\end{equation}
and this follows in a completely similar way, including the corresponding choices that make the 
adjoint ideals involved well-defined. After performing $(k-1)$ steps, we obtain the inclusion 
\begin{equation}\label{stepk-1}
\JJ\big(S, \parallel mM_S\parallel \big) \subseteq 
\Adj_{\Delta_{k-1,S}} (S,  \parallel mM + H + (k-2)(K_X + S) + \Delta_1 + \ldots + 
\Delta_{k-2}\parallel_{|S}).
\end{equation}
The usual (by now) Basic Lifting argument based on Proposition \ref{basic_lifting_2} shows that for this last ideal we have the inclusion 
\begin{equation}\label{sections}
H^0 (S, \OO_S (N) \otimes \Adj_{\Delta_{k-1,S}} (S,  \parallel mM + H + (k-2)(K_X + S) + \Delta_1 + \ldots + 
\Delta_{k-2}\parallel_{|S})) \subseteq $$
$$\subseteq H^0 (S, \OO_S (N) \otimes \frc_{|N|}), 
\end{equation}
where we recall that we denote $N= mM + H + (k-1)(K_X + S) + \Delta_1+\ldots 
+ \Delta_{k-1}$.
Lastly, we appeal to Lemma \ref{klt_inclusion}(1) in order to deduce the inclusion 
\begin{equation}\label{final}
\frc_{|N|} \subseteq 
\JJ \big( (S, \Delta_{0,S}); \parallel N\parallel _{|S} \big).
\end{equation}
Putting ($\ref{stepk-1}$), ($\ref{sections}$) and ($\ref{final}$) together, and making sure that we chose $H$ sufficiently positive so that 
$$\OO_S (mM_S + H_S + (k-1)K_S  + \Delta_{1,S} + \ldots + \Delta_{(k-1),S}) \otimes 
\JJ\big(S, \parallel mM_S\parallel \big) $$ 
is globally generated, we finally obtain ($\ref{main}$) precisely as in the previous steps. Note that in the process we also 
produced inductively, just as above, sections in $mM_S + H_S + (k-1)K_S + \Delta_{1,S} + \ldots + \Delta_{(k-1),S}$ which lift to $X$, so we have $S \not\subset \BB(N)$,
as promised at the beginning. This completes the induction step, hence the proof of Proposition \ref{intermediate}.

\noindent
\emph{Step 5.}
Now we come to getting rid of the ample line bundle $H$ from the previous steps. We can do this using the method of Takayama, given the stronger positivity assumption that the divisor $A$ is nef and big. First we recall the following useful observation.

\begin{lemma}[\cite{takayama}, Example 2.2(1); cf. also \cite{hm2}, Lemma 3.2]\label{local_inclusion}
Let $X$ be smooth, and $(X, \Lambda)$ a klt pair on $X$. Consider a line bundle $L$ on $X$, and $s \in H^0 (X, L)$, with $D : = Z(s)$. If $B$ is an effective $\QQ$-divisor  such that $B - D \leq \Lambda$, 
then $s \in H^0 (X, L \otimes \JJ(B))$. 
\end{lemma}

\noindent
For later use, we state a slightly more explicit version of Takayama's approach than what  is needed here. The proof entirely follows that provided in \cite{takayama}, p.568.

\begin{lemma}[Precise Takayama Lemma]\label{precise_takayama}
Consider a Cartier divisor $M$ such that $M \sim_{\QQ} k (K_X + S + A + \Delta)$, where $S$ 
is a smooth irreducible divisor and $A$ and $\Delta$ are effective $\QQ$-divisors on $X$.
Let $m$ be a positive integer, and $D\in |mM_S|$ defined by a section $s$.
Suppose $l$ is a positive integer. Let $H$ be an effective Cartier divisor on $X$ 
such that $S \not\subset {\rm Supp}(H)$. Assume that the
divisor $F_l : = lD+H_S$ can be lifted to $E_l \in |lmM+H|$, and that in addition the 
following two conditions are satisfied:

\noi
(1) We have $A + \Delta \sim_{\QQ} \frac{mk-1}{lmk} H + \Delta^\prime + B$, 
where $\Delta^\prime$ is an effective $\QQ$-divisor with $S \not\subset {\rm Supp}(\Delta^\prime)$ and 
$(S, \Delta^\prime_S)$ is klt,  and $B$ is ample.

\noi
(2) The pair $(S, \frac{mk-1}{lmk} H_S -\frac{1}{mk}D + \Delta^{\prime}_S)$ is klt.\footnote{If $H$ is ample and it moves, this condition is automatically satisfied if $H$ is general in its linear system and $l\gg0$.}  

\noi
Then $D$ can be lifted to a divisor in $|mM|$. 
\end{lemma}
\begin{proof} 

\noindent
Note first that condition (2) in the statement and Lemma \ref{local_inclusion} imply that
\begin{equation}\label{claim}
s \in H^0 (S, \OO_S(mM_S) \otimes \JJ ( S,  \frac{mk-1}{lmk}  F_l  + \Delta^{\prime}_S) ).
\end{equation}
Given ($\ref{claim}$), the fact that the section $s$ can be extended to a section of $\OO_X(mM)$ follows once more by Basic Lifting, Proposition \ref{basic_lifting_1}. Indeed, we have an exact sequence
$$0 \lra \JJ(X, \frac{mk-1}{lmk} E_l + \Delta^{\prime})(-S) \lra \Adj_S (X, \frac{mk-1}{lmk} E_l + \Delta^{\prime})
\lra \JJ(S, \frac{mk-1}{lmk} F_l + \Delta^{\prime}_S) \lra 0,$$
so it is enough to check the vanishing
$$H^1(X, \OO_X(mM-S)\otimes  \JJ (X, \frac{mk-1}{lmk} E_l + \Delta^{\prime}))  = 0.$$
But note that using condition (1) we can write
$$mM - S \sim_{\QQ} K_X  + \Delta^{\prime} + B + \frac{mk-1}{k} M +\frac{mk-1}{lmk}H .$$
We can then apply Nadel Vanishing, recalling that $E_l \in |lmM + H|$.
\end{proof}

We can now continue with the proof of the Theorem. By Proposition \ref{intermediate}, we know that there exists a fixed ample divisor $H$ on $X$, with $H_S$ given by $h_S = 0$, such that for every section $s \in H^0 (S, \OO_S(mM_S))$ and every integer $l$ sufficiently divisible, the section $s^l \cdot h_S$ extends to $X$. Denote $D : = Z(s)$, and $F_l:= lD + H_S = Z(s^l\cdot h_S)$. 
Consider also $E_l \in |mlM + H|$ such that ${E_l}_S = F_l$. In order to deduce that $D$ lifts to $X$, we need to check conditions (1) and (2) in Lemma \ref{precise_takayama}, for some integer $l$. Note first that  
the assumption that $A$ is big and nef and $S \not\subset \BB_{+} (A)$ implies that we have a decomposition $A \sim_{\QQ} A^{\prime} + C$, with $A^{\prime}$ ample, $C$ effective of arbitrarily small norm, and $S\not\subset {\rm Supp}(C)$. (We can rewrite the sum as 
$((1-\epsilon)A + \epsilon A^{\prime}) + \epsilon C$, with $\epsilon$ arbitrarily small.) 
If $l$ is large enough, we have that $B : = A^{\prime} -  \frac{mk-1}{mkl}H$ is still 
ample. Thus (1) in Lemma \ref{precise_takayama} is satisfied by taking $\Delta^\prime = \Delta  + C$. Since $(S, \Delta^{\prime}_S)$
is klt, for $l$ is sufficiently large (2) is also clearly satisfied.

\noindent
\emph{Step 6.}
We finally use a trick involving the previous steps in order to deduce the full statement of the Theorem 
from that involving $\BB(M_S)$ instead of $\BB_{-} (M_S)$ and $M$ being $\QQ$-effective instead of pseudo-effective.
Let us fix an ample Cartier divisor on $X$. Now pick any $m \ge 1$. If we replace $M$ by $ M + \frac{1}{m} H$, in other 
words $A$ by $A + \frac{1}{m}H$ in the main statement, all of the hypotheses are still satisfied, and in addition 
$\BB (M_S + \frac{1}{m}  H_S) \subseteq \BB_{-} (M_S)$ and $M + \frac{1}{m} H$ is now $\QQ$-effective. Thus we can apply what 
we proved in the previous steps to deduce the surjectivity of the map
$$H^0 (X, \OO_X (mM+ H)) \longrightarrow H^0 (S, \OO_S (mM_S+ H_S)).$$
But this can be done individually for every $m$, with this fixed $H$. (Up to here this is essentially the argument of
Corollary \ref{extension2}.) Now note that Takayama's argument in Step 5 can be applied one more time to this 
situation: since $M$ is of the correct form, and $H$ is fixed, the exact same argument shows that we can get rid of 
$H$ to deduce, for all $m$, the surjectivity of 
$$H^0 (X, \OO_X (mM)) \longrightarrow H^0 (S, \OO_S (mM_S)).$$
\end{proof}

\begin{remark}
We would like to note that some of the reduction arguments in Steps 1, 2 and 3 above are similar in nature to reductions allowing the use of extension theorems towards the proof of existence of flips, appearing in \cite{hm1}, while  Step 4 uses some ideas introduced in \cite{kawamata} and extended in \cite{hm1}.  
\end{remark}

\noindent
{\bf Special log-resolutions.}
In the proof above we used a statement on the existence of resolutions separating the divisors with log-discrepancy 
strictly less than $1$ with respect to a klt pair. Corollary \ref{separation} below was already proved in \cite{km} Proposition 2.36, and also in \cite{hm2} Lemma 6.5, where it is used for similar purposes. It is enough for what we need, and could 
be simply quoted here. We however to take to opportunity 
to include a sketch of the proof of a more explicit result, describing the nature of the divisorial valuations with log-discrepancy less than $1$, which is interesting in its own right and does not seem to appear in the literature in this generality.

\begin{theorem}\label{valuations}
Let $(X,\Delta)$ be a klt pair, not necessarily effective. Then there exists only a finite number of divisorial valuations $E$ 
such that $0 < {\rm ld} (E; X, \Delta) < 1$, and they correspond to weighted blow-ups at the closed subsets $Z\subset \tilde X$ 
with ${\rm mld}(\mu_Z; \tilde X, \tilde{\Delta}) < 1$, where $(\tilde X, \tilde{\Delta})$ is a log-resolution of $(X, \Delta)$. 
\end{theorem}
\begin{proof}
We appeal to the following statement in \cite{dfei}, Theorem 0.3: \emph{Let $X$ be an
$n$-dimensional smooth variety, $D = D_1+ \ldots +D_n$ a simple normal crossings divisor on $X$, 
and consider $Z$ a component of $D_1 \cap \ldots \cap D_m$ with $m \le n$. Let $x_1, \ldots, x_m$ be part of a 
system of parameters around the generic point of $Z$ such that $D_i = (x_i = 0)$. Let $E$ be a divisorial valuation centered at $Z$. Assume that there are positive integers $a_1, \ldots, a_m$ such that ${\rm val}_E (x_i) 
\ge a_i$ for all $1\le i\le m$. Then we have ${\rm ld}(E; X, \emptyset) \ge \sum_{i =1}^m a_i$, and equality holds if and only if $E$ is the exceptional divisor of the weighted blow-up along $Z$ with weights $a_1, \ldots, a_m$.}

This is stated in \cite{dfei} in the case of the standard coordinate system in $\CC^n$, but since a coordinate system determines locally an \'etale morphism $X \rightarrow \CC^n$, and log-discrepancies are preserved by such maps, the general statement follows immediately by base-change. The proof is based on the arc space methods in \cite{elm}.

To deduce the statement of the Theorem, we first reduce to the smooth and SNC case by passing to a log-resolution.  We write $\Delta = \sum_i d_i D_i$, 
with $d_i <1$, and assume that locally $D_i = (x_i = 0)$. We consider a center $Z$ of a divisorial valuation $E$ as in the statement. Writing ${\rm val}_E (x_i) = a_i$, we obtain
$${\rm ld} (E; X, \Delta) \ge \sum_i a_i - \sum_i d_i^\prime a_i = \sum_i (1-d_i^\prime ) a_i,$$
where the inequality comes from the fact stated above that ${\rm ld}(E; X, \emptyset) \ge \sum_i a_i$.
Here $d_i^\prime$ is either $d_i$ or $0$, and hence we want to impose  
$\sum_i (1-d_i^\prime ) a_i < 1$. Since the $a_i$ are positive integers, we must have all 
$d_i^\prime = d_i$, i.e. the center $Z$ is in fact an intersection of components of $\Delta$.
By the statement from \cite{dfei} above, if $E$ is not the exceptional divisor of a weighted blow-up, then we must in fact have ${\rm ld}(E; X, \emptyset) \ge \sum_i a_i + 1$, so by adding $1$ to the expression above we see that the log-discrepancy cannot be less than $1$. We have thus concluded that 
$E$ is must be the exceptional divisor of a weighted blow-up with weights $a_1, \ldots, a_n$ positive 
integers such that $0 < \sum_i (1-d_i) a_i < 1$, with $d_i < 1$ fixed rational numbers. This has only a
finite number of solutions in the $a_i$.
\end{proof}

\begin{corollary}\label{separation}
Let $(X, \Delta)$ be a klt pair. Then there exists a log-resolution on which all the divisors 
$E$ such that $0 < {\rm ld} (E; X, \Delta) < 1$ are disjoint.
\end{corollary}
\begin{proof}
Since by Theorem \ref{valuations} there exist only a finite number of divisorial valuations with log-discrepancy between $0$
and $1$, the corresponding divisors can be made disjoint after possibly blowing up and going to a higher model.
\end{proof}

\section{Results in the relative setting and for fibers of families}

Theorem \ref{extension0} has an obvious relative version, with essentially the same proof. We state it below for later use.

\begin{theorem}\label{relative}
Let $\pi : X \rightarrow T$ be a projective morphism from a normal variety $X$ to a normal affine variety $T$, and $S \subset X$ a normal irreducible Cartier divisor. Let $\Delta$ be an effective $\QQ$-divisor on $X$ such that $[\Delta] = 0$ and $S \not \subset {\rm Supp}(\Delta)$. Let $A$ be a $\pi$-nef and big $\QQ$-Cartier divisor such that $S \not\subseteq \BB_{+} (A)$, and $k$ a positive integer such that $M =  k(K_X + A + \Delta)$ is Cartier. Assume the following:
\begin{itemize}
\item $(X, S + \Delta)$ is a plt pair. 
\item $M$ is $\pi$-pseudo-effective.
\item the restricted base locus $\BB_{-} (M)$ does not contain any irreducible closed subset $W\subset X$ 
with ${\rm mld}(\mu_W ; X, \Delta + S) < 1$ which intersects $S$. 
\item the restricted base locus $\BB_{-} (M_S)$ does not contain any irreducible closed subset 
$W\subset S$ 
with ${\rm mld}(\mu_W ; S, \Delta_S) < 1$.
\end{itemize}
Then the restriction maps 
$$\pi_* \OO_X (mM)) \longrightarrow \pi_*\OO_S (mM_S)$$
are surjective for all $m \ge 1$.  Moreover, if $X$ is smooth, it is enough to assume that $M\sim_{\pi,\QQ} k(K_X + A + \Delta)$, with $M$ Cartier.
\end{theorem}

As mentioned in the Introduction, Theorem \ref{relative} can be used to prove a deformation-invariance-type statement, sometimes without an a priori (pseudo-)effectivity condition on $M$. The proof below was inspired by discussions with R. Lazarsfeld and M. Musta\c t\u a, and we thank them for allowing us to include these ideas.

\begin{proof} (\emph{of Theorem \ref{fibration}}.)
Fix a point $0\in T$. Upon allowing to shrink $T$ around $0$, it is standard (cf. \cite{positivity} p.314) that the result can be reduced to the 
following extension statement: assume all the singularity and transversality hypotheses in the statement only with 
respect to $X_0$. Then the restriction maps 
$$H^0 (X,  \OO_X (mM)) \longrightarrow H^0 (X_0, \OO_{X_0} (mM_{X_0}))$$
are surjective for all $m \ge 1$. 

Assuming then that $T$ is affine, it is clear that Theorem \ref{relative}  can be applied to deduce this when $M$ is pseudo-effective, so we only need to concentrate on part (ii). Assuming that $(X_0, k\Delta_{X_0})$ is klt, 
we need to show that the hypothesis of $M$ being pseudo-effective is not needed. Indeed, by carefully inspecting the proof of
Theorem \ref{extension0}, one notes that the hypotheses on $\BB_{-} (M_{X_t})$ and $\BB_{-} (M)$ (which could a priori be the whole $X$) are trivially satisfied: indeed, the only non-canonical centers that could affect the definition of any adoint ideal would have to 
come from $[k\Delta]$.
On the other hand, the pseudo-effectivity of $M$ was needed in the proof of Theorem \ref{extension0} in order to deduce two other things:
\begin{enumerate}
\item  $mM+ H$ is big for any $m\ge 1$ and any ample $H$, and thus 
vanishing theorems can be applied.
\item One can reduce to the case when no pseudo non-canonical centers of $\Delta$ are disjoint from the hypersurface $S$. 
\end{enumerate}
Item (2) can be easily dealt with in our situation by simply shrinking $T$, since no pseudo non-canonical center 
of $\Delta$ disjoint from $X_0$ can dominate $T$. Thus the key point is to show that in the case of 
fibrations $M$ is automatically $\pi$-pseudo-effective if $M_{X_0}$ is assumed to be so.\footnote{This is certainly the case by assumption if $\Delta \neq 0$, and is in any case necessary to make the result 
non-trivial.} 

To prove this, assume by contradiction that $M_{X_0}$ is pseudo-effective, but $M_{X_t}$ is not pseudo-effective for general $t \in T$. We can find a $\pi$-ample $\QQ$-divisor $B$ on $X$ such that $M_{X_t} + B_{X_t}$ is big for general $t\in T$, but arbitrarily close to the boundary of the effective 
cone of $X_t$. Recall now that one can define a volume function on numerical $\QQ$-divisor clasess
by 
$${\rm vol}(D) = \limsup_{m\to\infty}\frac{n!\cdot h^0(X, \OO_X(mD))}{m^n},$$ 
where the limit is taken over $m$ sufficiently divisible (cf. \cite{positivity} 2.2.C). This can be extended to a continuous function on the Neron-Severi space, which is $0$ on the boundary of the pseudo-effective cone (cf. \cite{positivity} Corollary 2.2.45). The continuity of this function 
implies then that for such a choice we have 
$$0 < \vol (M_{X_t} + B_{X_t}) \ll1.$$ 
Now the divisor $M + B$ is $\pi$-big and still satisfies the hypothesis of Theorem \ref{relative} (as the base locus can only get smaller), so this theorem and the reasoning above can be applied directly to deduce invariance with $t$ for the sections of $H^0 (X_t , \OO_{X_t}
(m (M_{X_t} + B_{X_t})))$, and hence
$$\vol (M_{X_0}  + B_{X_0}) =  \vol (M_{X_t}  + B_{X_t}), {\rm~for~general~}t.$$
On the other hand, since $M_{X_0}$ is pseudo-effective, on the central fiber we have that
$$\vol (M_{X_0}  + B_{X_0}) \ge \vol (B_{X_0}),$$
which is greater than some fixed strictly positive constant  (as $M_{X_0}$ is assumed to be outside of the pseudo-effective cone of 
$X_t$). This gives a contradiction.
\end{proof}

It is worth recording this last part of the proof as a result interesting in its own. Note that if the central fiber of the family 
is normal, than this is the case for all the fibers in a neighborhood.

\begin{corollary}\label{pseudoeffective}
Let $\pi : X \rightarrow T$ be a projective morphism from a variety $X$ to a smooth curve $T$. Let $\Delta$ be a horizontal effective 
$\QQ$-divisor on $X$ such that $[\Delta] = 0$. Let $A$ be a $\pi$-nef $\QQ$-divisor and $k$ a positive integer such that $M = k(K_X + A + \Delta)$ is Cartier. Denote by $X_t$ the fiber of $\pi$ over $t\in T$, and assume that $X_0$ is normal for some $0 \in T$. Assume the following:
\begin{itemize}
\item $(X, X_0 + \Delta)$ is a plt pair. 
\item $(X_0, k\Delta_{X_0})$ is a klt pair. 
\end{itemize}
Under these assumptions, if $M_{X_0}$ is pseudo-effective, then $M$ is 
$\pi$-pseudo-effective (i.e. $M_{X_t}$ is pseudo-effective for general $t\in T$). 
\end{corollary}
\begin{proof}
Note that we are only  assuming $A$ to be nef. One way is to simply note that the final paragraph in the proof of Theorem 
\ref{fibration} works also in this case, since we are adding the ample line bundle $B$. Alternatively, by Theorem \ref{fibration}, 
we have invariance of volumes after adding any arbitrarily small $\pi$-ample divisor to $M$, which says the same thing. 
\end{proof}

This applies to the invariance of part of the boundary of the pseudo-effective cone, just as Siu's original extension results imply the invariance of the Kodaira energy of a $\pi$-ample divisor.\footnote{Recall that the Kodaira energy of a big $\QQ$-divisor $D$ is ${\rm inf} \{ t\in \QQ ~| ~K_X + t D {\rm~ is~ pseudoeffective}\}$.}

\begin{corollary}\label{boundary}
Under the assumptions of Corollary \ref{pseudoeffective}, if $M_{X_0} \in \partial \overline{{\rm Eff}} (X_0)$, then 
$M_{X_t} \in \partial \overline{{\rm Eff}} (X_t)$ for $t\in T$ general. If in addition $A$ is assumed to be $\pi$-big, then the 
converse is also true.
\end{corollary}
\begin{proof}
By the previous Corollary and semicontinuity, we have that $M_{X_0} \in  \overline{{\rm Eff}} (X_0)$ if and only if $M_{X_t} \in 
\overline{{\rm Eff}} (X_t)$ for $t$ general. The statement follows immediately from the fact that Theorem \ref{fibration} implies 
$\vol (M_{X_0}  + B_{X_0}) =  \vol (M_{X_t}  + B_{X_t})$ for general t and any $\pi$-ample $\QQ$-divisor $B$, and also 
$\vol (M_{X_0}) =  \vol (M_{X_t})$ in case $A$ is $\pi$-nef and big (it can be easily checked that 
we cannot have $X_0 \subseteq \BB_{+} (A)$).
\end{proof}

\section{Some remarks on finite generation}

In this section we observe that the main extension result proved here, Theorem \ref{extension0}, and Takayama's technique described in Lemma \ref{precise_takayama}, combine to provide a quick proof 
of part of the finite generation statements involved in the Hacon-M$^c$Kernan  proof of existence of flips. As the new \cite{hm3} has become available, we can point to the fact that the result below is very similar to Theorem 6.2 there. The main idea is in any case due to Hacon and M$^c$Kernan, the contribution here being to streamline the argument using the tools mentioned above.  

Let $(X, \Delta)$ be a log-pair, with $X$ a normal projective variety and $\Delta$ an effective $\QQ$-divisor with $[\Delta] = 0$. Let $S \subset X$ be an irreducible normal effective Cartier divisor with $S \not\subset {\rm Supp}(\Delta)$, such that $(X, S +  \Delta)$ is a plt pair.
Consider $A$ a big and nef $\QQ$-divisor on $X$ such that $S \not\subseteq \BB_{+} (A)$. Assume that there exists a positive integer
$k$ such that $M = k(K_X + S + A + \Delta)$ is Cartier. Finally, assume that $S\not\subseteq \BB(M)$.
The interest is in understanding the restricted algebra $R_S$ given by the direct sum over $m$ of the images of the maps 
$$H^0 (X, \OO_X (mM)) \longrightarrow H^0 (S, \OO_S (mM_S)).$$ 

We first apply reduction steps precisely as in Steps 1 and 2 of the proof of Theorem
\ref{extension0}, to reduce to the case when $X$ is smooth, $S + \Delta$ has SNC support, and the irreducible components of the support of $\Delta$ 
which intersect $S$ are disjoint. Let us denote these by $E_1,\ldots, E_r$, and denote also $F_i : = E_{i, S}$. Thus the only log-canonical centers of $(S, \lceil \Delta_S\rceil)$ 
are the $F_i$'s, and the only log-canonical centers of $(X, \lceil \Delta \rceil)$ which 
intersect $S$ are the $E_i$'s. The assumption $S \not \subseteq 
\BB_{+} (A)$ implies as usual, up to slightly modifying $\Delta$ (note that we are not making any transversality assumptions) that we can assume that $A$ is in fact ample, and we will do so in what follows.

Since the argument below is quite  technical, we again first include a brief oultine. Unlike in the extension theorems in previous sections, in the situation above
there are no transversality assumptions. We start by adding a small amount of positivity to our divisor, which has the effect of decreasing the order 
of vanishing of the restricted linear series along the $F_i$. At the same time we ensure by passing to a suitable birational model that the 
asymptotic order of vanishing of this new linear series along the $E_i$ is the same of that  of its restriction to $S$ along the $F_i$. These choices allow us to 
obtain transversality and apply extension via Theorem \ref{extension0}, after subtracting the $E_i$ with appropriate coefficients so that they disappear either from 
$\Delta$ or from the stable base locus of the new series. (A similar idea already appears in \cite{el1} \S6,  where the extra positivity is provided by the Seshadri constant.)
Finally, one gets rid of the extra positivity added initially by using the precise version of Takayama's lemma \ref{precise_takayama}.

Moving to details, we denote $\delta_i : = {\rm ord}_{E_i} (\Delta) = {\rm ord}_{F_i} (\Delta_S)$.
We consider the asymptotic order of vanishing of the restricted (to $S$) linear series associated $K_X + S + \Delta + A$ along $F_i$:
$$c_i := \ord_{F_i} \parallel K_X + S + \Delta + A \parallel_S =  \ord_{F_i} \parallel \frac{M}{k} \parallel_S ~\in \RR.$$
Define the $\RR$-divisor on $S$
$$F: =  \sum_{i=1}^r \min \{\delta_i, c_i\} F_i.$$
Note that the asymptotic order of vanishing of $K_X +  S + \Delta + A$ along $E_i$ may 
a priori be smaller than $c_i$. Hence we could not directly apply extension via Theorem \ref{extension0} to $M - k \sum_{i=1}^r \min \{\delta_i, c_i\} E_i$ even if this were to be a 
$\QQ$-divisor, since the $E_i$'s may still be contained in $\BB_{-} (M)$. We arrange things as  follows, in order to apply extension to a modified divisor. We fix a positive integer $m$, to be specified  later. Choose a small positive 
number $\epsilon$ such that $A + mk\epsilon  (K_X +S + \Delta) {\rm~is ~ample}$.
A small calculation shows that this implies 
$$d_i := \ord_{F_i} ||K_X+\Delta +S +(1+\frac{1}{mk})A||_S < (1-\epsilon)c_i.$$
To work with something rational we consider, for a positive integer $l$, the usual order 
of vanishing
$$e_i :=\frac{1}{lmk}\ord_{F_i} |lmk(K_X+S+\Delta +(1+\frac{1}{mk}) A)|_S = 
\frac{1}{lmk}\ord_{F_i} |l(mM  + A)|_S.$$
If $l$ is sufficiently large, we have as in Step 3 of the proof of Theorem \ref{extension0}
that $d_i \le e_i < (1 - \epsilon) c_i$.

\begin{theorem}\label{rational}
In the setting and notation above, for every integral divisor $B$ such that 
$$mk\sum_{i=1}^r \min \{\delta_i, (1-\epsilon)c_i\} F_i \le B \le mkF,$$
we have that $|mM_S - B| + B =  |mM|_S$.
In particular, if $F$ is a $\QQ$-divisor and $m$ is such that $mkF$ is integral, then
$$|mM_S - mkF| + mkF =  |mM|_S.$$
\end{theorem}
\begin{proof}
We can perform one more birational modification, in order to be able to assume that
the orders of vanishing of the new linear series along $E_i$ are the same as those 
of their restrictions along the $F_i$. Namely, we consider a partial log-resolution 
$f : Y \rightarrow X$ of the linear series $| l(mM+ A)|$ as a composition of blow-ups along the codimension $2$ subvarieties $F_i$, in order to have basepoint-freeness at the generic points of codimension $2$ loci (precisely 
as in the well-known argument that base loci on surfaces can be resolved by blowing up points).
In other words (and since the $F_i$ are codimension $1$ in $S$), we can achieve the following:

\noi
(1) $T$, the proper transform of $S$ in $Y$, is isomorphic to $S$.
\\
\noi
(2) $f^*(K_X+S+\Delta) = K_Y +\Delta_1 +T$,
where $\Delta_1$ is an effective $\QQ$-divisor with SNC support
and $[\Delta_1] =0$. There exists an unique irreducible component of the support
of $\Delta_1$, which we continue to denote abusively by $E_i$, such that 
$E_i \cap T = F_i$.
\\
\noi
(3) The  base ideal of $|f^*(l(mM+A))|$ is of the of the form
$\OO_Y (lmk(\sum_{i=1}^r e_i F_i))$ at each of the generic points of the $F_i$. 

All linear series remain the same after performing this birational modification, so by switching back to the notation on $X$, we can thus assume in addition that 
$d_i = \ord_{E_i} ||K_X+\Delta +S +(1+\frac{1}{mk})A||$, which allows us to apply extension. Set 
$$\Gamma : = lmk(\sum_{i=1}^r \min\{e_i, \delta_i\} E_i) 
{\rm ~and~} M^{\prime} := l (mM+ A) - \Gamma.$$
This choice implies that no $E_i$ is contained in ${\rm Bs} |l(mM+ A)|$, and the same holds for $F_i$ 
on $S$, with respect to $|l(mM_S + A_S)|$. Since $M^\prime$ is still of the form required in Theorem \ref{extension0}, we can apply this result to deduce that the restriction map
\begin{equation}\label{useful_extension}
H^0 (X, \OO_X(M^\prime)) \longrightarrow H^0 (S, \OO_S(M^\prime_S))
\end{equation}
is surjective. We use this to prove the following:
\emph{If $B$ is an effective integral divisor on $S$ such that 
$B \ge mk \sum_{i=1}^r \min \{\delta_i, (1-\epsilon)c_i\} F_i$, then every divisor 
in the linear series $|mM_S - B| +  B$ can be lifted to $X$.}

To this end, consider a divisor $D \in |mM_S|$ that can be written 
as  $D = D_1 + B$ with $D_1$ an effective divisor in $|mM_S - B|$. Note that the hypothesis 
on $B$, together with the fact that $e_i < (1-\epsilon) c_i$, implies that $lB > \Gamma_S$. We  could 
choose from the very beginning $l$ sufficiently large so that $l A$  is very ample, and so for $H \in |l A|$ general, $(S,  \frac{mk-1}{lmk} H_S + \Delta_S)$ is klt. 
Putting all of this together, we can use (\ref{useful_extension}) in order to deduce that the divisor $lD + H_S = l D_1 + lB + H_S$ can be lifted to $X$. Finally, it is immediate to check that the two conditions
in the precise version of Takayama's lifting result, 
Lemma \ref{precise_takayama}, are satisfied. This implies that $D$ itself can be lifted to $X$. 

The argument above shows the inclusion $|mM_S - B| + B \subseteq  |mM|_S$.
On the other hand, if $B \le mkF$, the definition of $F$ implies the opposite inclusion. Indeed,  
we have that $F$ is contained in the stable base locus of the system of restricted linear series 
$\{|lM|_S\}_l$. Hence we obtain 
$$|mM_S - B| + B =  |mM|_S.$$
Finally, if $F$ is a $\QQ$-divisor, and $m$ is chosen such that  $mkF$ is integral, we can apply the above with $B = mkF$.
\end{proof}

\begin{remark}
If $F$ in the proof above is a $\QQ$-divisor (which happens for example when $\delta_i \le c_i$ for all $i$), assuming by induction simply that finite generation holds in dimension $n-1$, in particular on $S$, plus the well-known fact that is enough to check finite generation for a Veronese subalgebra, we obtain that the restricted algebra $R_S$ is finitely generated. This statement is enough to conclude  the existence of $PL$-flips,  hence the existence of flips, according to the method of Shokurov explained in \cite{hm2} and \cite{hm3}. In general it need not a priori be the case that $F$ is rational. Hacon and 
M$^c$Kernan provide a rather quick key argument in \cite{hm3}, Theorem 7.1, using diophantine approximation and crucially a stronger MMP assumption in dimension $n-1$ following from \cite{bchm}, in order to show that the coefficients must indeed be rational.
\end{remark}

\end{document}